\documentclass[11pt,reqno,a4paper]{amsart}
\usepackage[english,activeacute]{babel}
\usepackage{amssymb,amsmath,amsthm,amsfonts,mathrsfs,hyperref,mathtools,stmaryrd}
\usepackage[font=footnotesize]{caption}

\usepackage{tikz}
\usetikzlibrary{intersections,shapes,trees,arrows,spy,positioning,decorations,arrows.meta,decorations.pathreplacing,patterns,calc}
\usepackage{xcolor}

\definecolor{mylightgreen}{RGB}{218, 227, 25}
\definecolor{mygreen}{RGB}{117, 208, 84}
\definecolor{mylightblue}{RGB}{32, 163, 134}
\definecolor{mypetrol}{RGB}{42, 120, 142}
\definecolor{mydarkblue}{RGB}{65, 68, 135}
\definecolor{mypurple}{RGB}{68, 1, 84}

\usepackage{enumitem}  
\usepackage[T1]{fontenc} 
\usepackage[foot]{amsaddr} 
\usepackage{cleveref}

\usepackage{dsfont}
\usepackage{subcaption}

\theoremstyle{plain}
\newtheorem{theorem}{Theorem}[section]
\newtheorem{lemma}[theorem]{Lemma}
\newtheorem{proposition}[theorem]{Proposition}

\theoremstyle{definition}

\theoremstyle{remark}
\newtheorem{remark}{Remark}[section]

\numberwithin{equation}{section}

\newcommand{\Eb}  {{\mathbb E}}
\newcommand{\Ebf}  {{\mathbf E}}
\newcommand{\Nb}  {{\mathbb N}}
\newcommand{\Pb}  {{\mathbb P}}
\newcommand{\Pbf}  {{\mathbf P}}
\newcommand{\Qb}  {{\mathbb Q}}
\newcommand{\Rb}  {{\mathbb R}}

\newcommand{\Ubf}  {{\mathbf U}}

\newcommand{\bM}  {{\overline{M}}}
\newcommand{\bX}  {{\overline{X}}}
\newcommand{\bT}  {{\overline{T}}}
\newcommand{\btau}  {{\overline{\tau}}}

\newcommand{\Fs} {{\mathcal F}}

\newcommand{\dd}{{\rm d}}
\newcommand{\1}{{\mathds{1}}}

\newcommand{\K}{{R}}

\newcommand{\del}{{\vartheta}}

\def\llam{\lambda\hspace{-5.1pt}\lambda}
\def\eps{\varepsilon}

\allowdisplaybreaks[4]
 
\title[]{A two-size Wright--Fisher model: asymptotic analysis via uniform renewal theory}

\author{G. Alsmeyer$^{1}$}
\email{gerolda@uni-muenster.de}

\author{F. Cordero$^{2}$}
\email{fernando.cordero@boku.ac.at}

\author{H. Dopmeyer$^{3}$}
\email{hannah.dopmeyer@uni-bielefeld.de}

\address{$^{1}$Institute for Mathematical Stochastics, Department of Mathematics and Computer Science, University of M\"unster, Orleans-Ring 10, 48149 M\"unster, Germany}
\address{$^2$Institute of Mathematics, BOKU University, Gregor-Mendel-Stra\ss e 33/II, 1180 Vienna, Austria}
\address{$^3$Faculty of Technology, Bielefeld University, Box 100131, 33501 Bielefeld, Germany}

\date{\today}

\begin{document}

\begin{abstract}
We consider a population with two types of individuals, distinguished by the resources required for reproduction: type-$0$ (small) individuals need a fractional resource unit of size $\vartheta \in (0,1)$, while type-$1$ (large) individuals require $1$ unit. The total available resource per generation is $R$. To form a new generation, individuals are sampled one by one, and if enough resources remain, they reproduce, adding their offspring to the next generation. The probability of sampling an individual whose offspring is small is $\rho_{R}(x)$, where $x$ is the proportion of small individuals in the current generation. We call this discrete-time stochastic model a two-size Wright--Fisher model, where the function $\rho_{R}$ can represent mutation and/or frequency-dependent selection. We show that on the evolutionary time scale, i.e. accelerating time by a factor $R$, the frequency process of type-$0$ individuals converges to the solution of a Wright--Fisher-type SDE. The drift term of that SDE accounts for the bias introduced by the function $\rho_{R}$ and the consumption strategy, the latter also inducing an additional multiplicative factor in the diffusion term. 
To prove this, the dynamics within each generation are viewed as a renewal process, with the population size corresponding to the first passage time $\tau(R)$ above level $R$. The proof relies on methods from renewal theory, in particular a uniform version of Blackwell's renewal theorem for binary, non-arithmetic random variables, established via $\varepsilon$-coupling.
\end{abstract}

\maketitle


 \setcounter{tocdepth}{1}

\section{Introduction}
In population genetics, the Wright--Fisher model \cite{Fisher:1930, Wright:1931} is one of the most prominent and~widely used models in discrete time when aiming to describe the evolution of the type composition of individuals under the influence of evolutionary forces such as mutation, selection, gene flow and environmental changes. In its original form, the model considers the dynamics of a population of genes of two allelic types under neutral selection,
but it has by now been generalized in many directions, see e.g. \cite{CGS+,GS20,GS17, GSW23} for recent work of this kind.

In most of these generalizations, it is typically assumed that the population consists of a fixed number of haploid individuals that reproduce asexually. However, differences in consumption strategies within the population are often overlooked. This assumption can be reasonable in cases where all individuals follow the same consumption strategy, or when the resources required for reproduction are negligible in comparison to the total resource pool. For instance, the standard Wright--Fisher model falls into the first category if we interpret the fixed population size as the constant amount of available resources, assuming that each individual requires exactly one unit of resource to reproduce and is selected at random with replacement.

From a biological perspective, a classic framework addressing reproductive resource trade-offs is the \emph{$r$- and $K$-selection theory} introduced in \cite{MW:01}. It postulates a trade-off between offspring quantity and quality: $r$-strategists (e.g., mice) produce many offspring at low cost, while $K$-strategists (e.g., elephants) produce fewer, costlier offspring. The $r$-strategy is favored in unstable environments; the $K$-strategy in stable ones. In a different context, \cite{Tilman:1982} examined the role of resource use in species coexistence. There, species capable of surviving with minimal resources are favored over those with higher demands, a principle known as the \emph{$R^*$-rule}.

Recently, Gonz\'alez Casanova et al. (2020) \cite{GMP20} proposed an extension of the Wright--Fisher model with selection, incorporating reproductive costs to explore their evolutionary consequences. In their framework, the population consists of two types of individuals, where reproductive success is influenced by both genic selection and a resource consumption strategy. Notably, these consumption strategies depend solely on the type of the individual. This approach provides a stepping stone toward understanding how resource allocation in reproduction can influence evolutionary fitness, potentially in ways that deviate from traditional biological models.

More specifically, it is assumed in \cite{GMP20} that individuals either require a fractional resource unit $\del\in (0,1)$ or a full resource unit $1$ to reproduce and that individuals requiring $1$ resource unit per reproduction have a selective advantage. This differentiation leads to two distinct reproductive strategies based on resource consumption, which can influence the evolutionary trajectory of the population. Mutation or more complex forms of selection are not considered. As a result of this framework, the population size is no longer fixed but becomes a stochastic variable that fluctuates over time (see \cite{GMP20} for comparison to other models with variable population sizes). The large population limit of the model treated in \cite{GMP20} is dual to a branching process with interaction. This duality relation is a special case of a more general result stated in \cite[Theorem 2]{GP21}. We refer to \cite{OP24+} and \cite{CGP24} for further discussion on branching processes with interactions, where the implications of stochastic fluctuations in population size and their impact on evolutionary processes are explored in more depth.

The objective of this paper is to study an extension of the model described above by incorporating general forms of (frequency-dependent) selection and mutation. More precisely, we consider a discrete-time, finite population model with a fixed amount $\K \in \Rb_{+}$ of resources available for reproduction in each generation. These are consumed each time a new individual is produced. There are two types of individuals, called type-$0$ and type-$1$, whose distinguishing feature is the amount of resources needed to produce them. Namely, it takes a fraction $\del\in(0,1)$ of resource units to produce a type-$0$ individual and one unit of resources to produce a type-$1$ individual. To form a new generation, the individuals are sequentially sampled (with replacement) from the current generation,  each time subtracting the required amount of resources from those still available. This continues until the quantity $\K$ is completely used up or exceeded, where one can imagine that in the latter case the missing quantity of resources for the production of the last individual is taken from an internal storage. The amount of resources needed to produce a new generation may therefore exceed $\K$ (but not $\K+1$). We usually interpret the two types as two sizes (see Remark \ref{remark:interpretation} and Figure \ref{fig.illustration}) and therefore refer to this extension as the \textit{two-size Wright--Fisher model}. We examine the evolutionary impact of different consumption strategies on the frequency process of type-$0$ individuals within this framework, developing a new method for a comprehensive treatment, based on renewal theory. The limiting process is characterized as the solution of a stochastic differential equation (SDE).  It is also worth noting that in the special case where our model coincides with that of \cite{GMP20}, our main theorem shows that the drift term derived in \cite[Theorem 1]{GMP20} is not correct, as will be further explained after the statement of our result (see Remark \ref{rem:main}). A similar family of SDEs was considered in \cite{Gillespie:1974}, which studied the selective effects of within-generation variance on the offspring number, see Remark \ref{remark gillespie}.

A more formal introduction of the model is provided in Section \ref{subsec:model}, but already at this point the connection of the one-step dynamics to renewal theory is evident: the amount of consumed resources equals the sum of iid positive random variables (taking values $\del$ or $1$), and hence a renewal process with Bernoulli-type increments. The (varying) population size corresponds to the first passage time $\tau(\K)$ above level $\K$ of this renewal process, see Section \ref{subsec:connection} for details. 

\vspace{.2cm}
Our main goal is to establish the asymptotic behavior of the frequency process as the amount of resources $R$ tends to infinity. Let $X_{t}^{\K}$, $t\geq0$, denote the proportion of type-$0$ individuals at time $t$ in the model with available resources $\K$. Our main result, Theorem \ref{main}, states that, under suitable conditions, the time-scaled process $\big(X_{\lfloor \K t\rfloor}^{\K}\big)_{t\geq 0}$ converges, as $\K\to\infty$, to the solution of the SDE
$$ \dd X_{t}\ =\ \big(-(1-\del) X_{t}(1-X_{t}) + \rho(X_{t})\big)\,\dd t + \sqrt{X_{t}(1-X_{t})\big(1-(1-\del) X_{t}\big)}\, \dd B_{t}, $$
where $B$ denotes a standard Brownian motion and $\rho:[0,1]\to\Rb$ is an appropriate Lipschitz function summarizing the effects of selection and mutation. Note that the size parameter $\del$ affects both the drift term and the diffusion term of the SDE. We will refer to the solution of this SDE as the \emph{two-size Wright--Fisher diffusion}.

The result will be obtained by showing uniform convergence of the respective generators. Owing to the connection to renewal theory, this uniform convergence leads to certain uniform renewal-type convergence results for random walks with Bernoulli-type increments, as indicated above. We establish uniform versions of the elementary renewal theorem and a uniform version of Blackwell's renewal theorem. These versions cannot be deduced from existing uniform renewal theorems in the literature (see \cite{BorovkovFoss:1999,Lai:1976}) and may therefore be of interest in their own right.

An obvious modification of the model is to complete a generation with the last individual whose reproduction costs are fully covered by the remaining resources. This variant is shortly described in Subsection \ref{sec:variant} together with a statement of a counterpart of Theorem \ref{main}.

We have organized our work as follows. Section \ref{sec:main} introduces the model in detail, states our main result (Theorem \ref{main}) and also explains the connection to renewal theory. In Section \ref{sec:URT}, the required uniform convergence results from renewal theory are established, which then allow us to give the proof~of Theorem \ref{main} in Section \ref{sec:proof main}. We finish with two sections providing short discussions of the afore-mentioned model variant (Section \ref{sec:variant}) and of the two-size Wright--Fisher diffusion (Section \ref{sec:asympt}).

\section{Main Result} \label{sec:main}
\subsection{The model and its large population limit}\label{subsec:model}

We consider the evolution in discrete time (generations) of a finite population of haploid individuals, each of which can be of type $0$ or type $1$. A constant amount $R \in \Rb_{+}$ of resources is available in each generation, reserved for reproduction and consumed in the formation of the next generation. The cost (in resource units) of placing an offspring of type $0$ into the next generation is $\del \in (0,1)$, while the cost for an offspring of type $1$ is $1$. If the relative proportion of type-$0$ individuals in the current generation is $x \in [0,1]$, the next generation is formed by sequentially sampling (with replacement) from the current generation. Each sampled individual produces an offspring according to the following rules.
\begin{enumerate}\itemsep3pt
\item If $k \in \Nb_{0}$ individuals of types $r_1, \ldots, r_k$ have already been placed in the next generation at cost
$$
R_k \coloneqq \sum_{i=1}^{k} \big((1 - r_i)\del + r_i\big)
$$
(with $R_k \coloneqq 0$ if $k = 0$), and if $R_k < R$, the $(k+1)$-th individual is created as follows. First, a parent is randomly selected according to fitness: the parent is of type $\hat{r}_{k+1} = 0$ with probability $s_{R}(x)$ and of type $\hat{r}_{k+1} = 1$ with probability $1 - s_{R}(x)$. It is natural to assume that $s_{R}(0) = 0$ and $s_{R}(1) = 1$. The selected parent then produces the $(k+1)$-th individual, which mutates to type $i \in \{0,1\}$ with probability $\beta_{i,R}$ (so $r_{k+1} = i$), or retains the parental type with probability $1 - \beta_{0,R} - \beta_{1,R}$ (so $r_{k+1} = \hat{r}_{k+1}$).

\item If $R_{k+1} < R$, replace $k$ by $k+1$ and return to Step (1). Otherwise, if $R_{k+1} \geq R$, the reproduction process terminates, and the next generation consists of the first $k+1$ individuals.
\end{enumerate}

The function $s_{R}$ models frequency-dependent selection, and the constants $\beta_{0,R}$ and $\beta_{1,R}$ represent mutation probabilities. Mutations are parent-independent, so silent events are allowed, where an individual of type $i$ mutates to type $i$. By construction, the probability of selecting a parent that places an offspring of type $0$ in the next generation is
\begin{equation}\label{sel-mut}
\rho_{\K}(x)\coloneqq s_{R}(x)\big(1-\beta_{1,R}\big)+\big(1-s_{R}(x)\big)\beta_{0,R}.
\end{equation}
The function $\rho_{R}:[0,1]\to[0,1]$ will play a key role in our analysis, as will become apparent in Section~\ref{subsec:connection}. In fact, except for Remarks~\ref{rem:main} and~\ref{special cases}, our results will be stated in terms of the function $\rho_{R}$, without assuming the particular form~\eqref{sel-mut}. Moreover, the next remark shows that any function $\rho_{R}:[0,1]\to[0,1]$ that attains its maximum and minimum at the boundary points $0$ and $1$ can always be interpreted as resulting from a combination of frequency-dependent selection and mutation.

\begin{figure}[h]
\centering 
\begin{subfigure}{0.49\textwidth}
\includegraphics[width=0.99\textwidth]{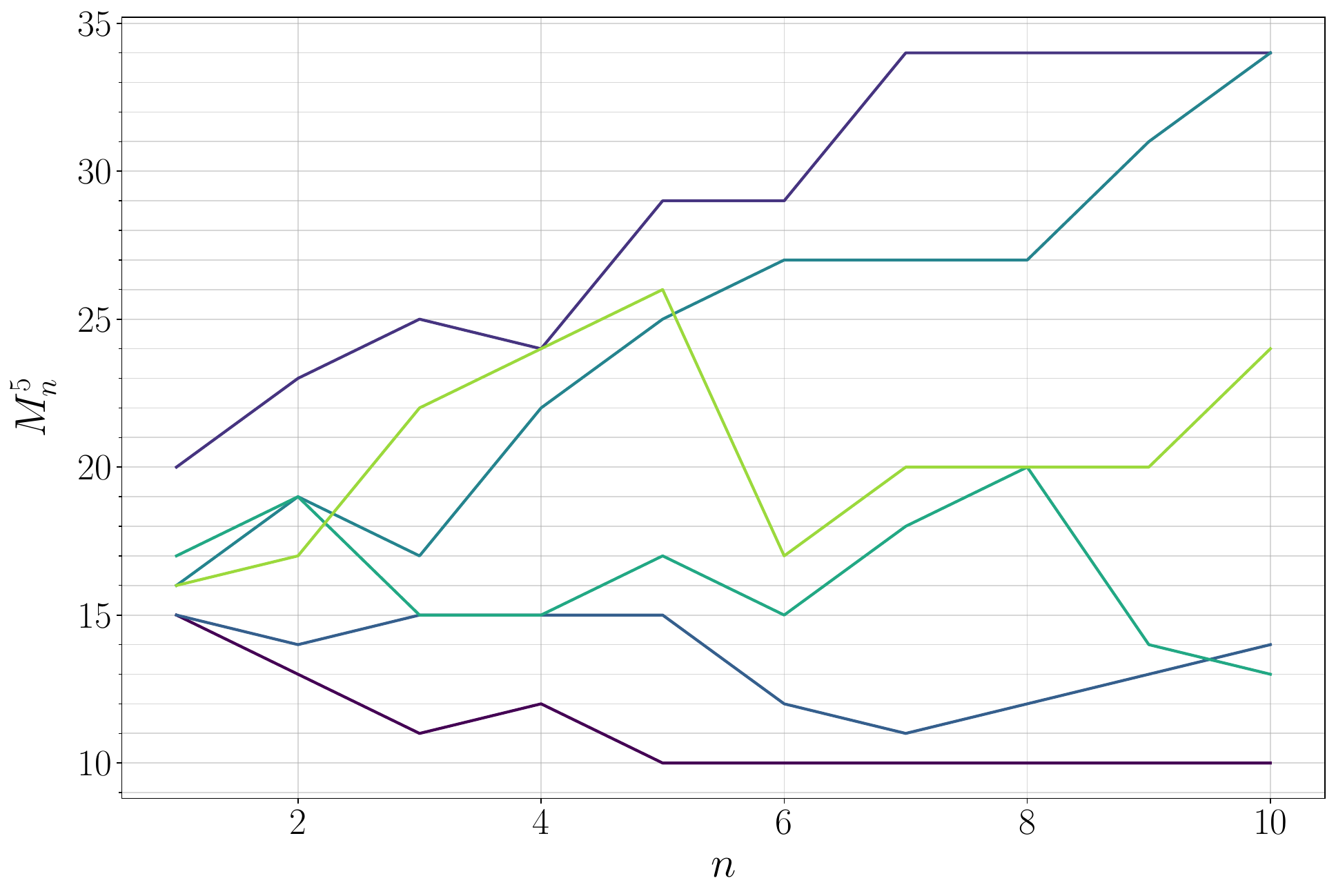}
\end{subfigure}
\begin{subfigure}{0.49\textwidth}
\includegraphics[width=0.99\textwidth]{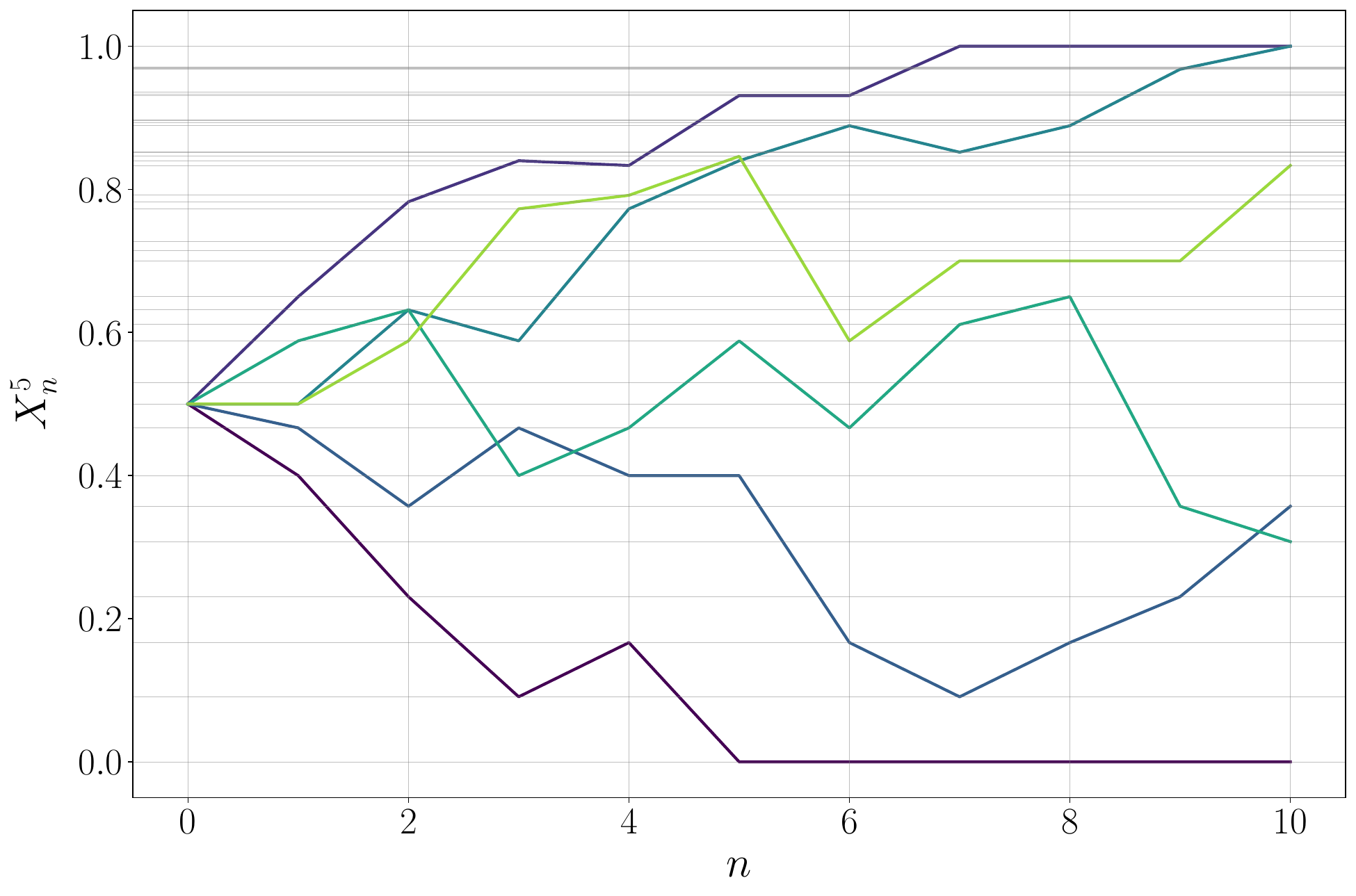}
\end{subfigure}
\caption{Simulation of the population size $M_{n}^R$ (\textit{left}) and the proportion of type-$0$ individuals $X_{n}^R$ (\textit{right}) for six populations of a two-size Wright--Fisher model with parameters $\K=5$, $\del=0.3$, $\rho_{R}(x)=x$ and $x_{0}=0.5$. The horizontal gray lines show the respective codomain.}
\label{simulation1}
\end{figure}

\begin{remark}\label{formofrhoR}
Although the selection-mutation mechanism underlying Eq.~\eqref{sel-mut} is biologically intuitive, our analysis is not restricted to that specific form of the function $\rho_{R}:[0,1]\to[0,1]$. However, if we are given a function $\rho_{R}$ such that
$$
\rho_{R}(0) \wedge \rho_{R}(1) \leq \rho_{R}(x) \leq \rho_{R}(0) \vee \rho_{R}(1) \quad \text{for all } x \in [0,1],
$$
we can always express it in the form~\eqref{sel-mut} by setting
$$
\beta_{0,R} \coloneqq \rho_{R}(0), \quad \beta_{1,R} \coloneqq 1 - \rho_{R}(1), \quad \text{and} \quad s_{R}(x) \coloneqq \frac{\rho_{R}(x) - \rho_{R}(0)}{\rho_{R}(1) - \rho_{R}(0)},
$$
provided that $\rho_{R}(0) \neq \rho_{R}(1)$. In the case where $\rho_{R}(0) = \rho_{R}(1)$, the function $s_{R}$ can be chosen arbitrarily. Note that the assumptions made for the function $\rho_{R}$ imply that $\beta_{0,R},\beta_{1,R}\in [0,1]$, $s_{R}:[0,1]\to[0,1]$, $s_{R}(0)=0$, and $s_{R}(1)=1$.
\end{remark}

\smallskip

Note that the population size is random (as $ \del<1$) and varies from generation to generation. Let $M_{n}^{\K}$ denote the population size in generation $n\in\Nb$ and $X_{n}^{\K}$ the relative proportion of type-$0$ individuals in this generation. Initially, there are $M_{0}^{\K}=m_{0}$ individuals and a relative proportion $X_{0}^{\K}=x_{0} \in [0,1]$ of type-$0$ individuals. We refer to this model as the \emph{two-size Wright--Fisher model with frequency-dependent selection and mutation}. 
Figure \ref{simulation1} shows a realization of the model. 

\begin{remark}\label{remark.eff_pop_size}
   Assume $X^R_{n}=x$ and that $R$ is large. From the construction of the model it is easy to see that the population size satisfies
\begin{equation*}
   M_{n}^R= \frac{R}{1-(1-\del)x}+O(1).
\end{equation*}
\end{remark}

\begin{remark}\label{remark:interpretation}
We typically interpret types as sizes (lengths), where individuals can be of length $\del$ or length $1$, and refer to them as \textit{small} (type $0$) or \textit{large} (type $1$). During the reproduction step, new individuals are added to the population until the total size reaches $\K$, which corresponds to the space available per generation.
\end{remark}

Figure \ref{fig.illustration} illustrates a reproduction step in this model with the above interpretation in mind.

\begin{figure}[h!] \centering
\begin{tikzpicture}
		\draw[opacity=1,thick] (0,0) -- (0,2); 
	
		\node  at (0,2.2) {\footnotesize $0$}; 
		
		\node[align=right] at (-0.3,1.45) { \footnotesize $n$}; 
		\node[align=right] at (-0.5,0.6) {\footnotesize $n+1$};
		\draw[draw=black, thick, fill=mydarkblue, postaction={
        pattern=grid, pattern color = black}, opacity=1] (0,1.2) rectangle ++(2.1,0.5); 
		\draw[draw=black, thick, fill=mygreen, postaction={
        pattern=dots, pattern color = black}] (2.1,1.2) rectangle ++(0.8,0.5);
		\draw[draw=black, fill=mypurple,  thick] (2.9,1.2) rectangle ++(0.8,0.5);
		\draw[draw=black, fill=mypetrol, postaction={
        pattern=north east lines, pattern color = black}, thick] (3.7,1.2) rectangle ++(2.1,0.5);
		\draw[draw=black, fill=mylightgreen, postaction={
        pattern=crosshatch, pattern color = black}, thick] (5.8,1.2) rectangle ++(0.8,0.5);
		\draw[draw=black, fill=mylightblue, postaction={
        pattern=north west lines, pattern color = black}, thick] (6.6,1.2) rectangle ++(0.8,0.5);

		\draw[draw=black,  fill=mylightblue, postaction={
        pattern=north west lines, pattern color = black}, thick] (0,0.4) rectangle ++(0.8,0.5);
		\draw[draw=black, thick, fill=mypetrol, postaction={
        pattern=north east lines, pattern color = black},] (0.8,0.4) rectangle ++(2.1,0.5);
		\draw[draw=black, fill=mylightgreen, postaction={
        pattern=crosshatch, pattern color = black},thick] (2.9,0.4) rectangle ++(2.1,0.5);
		\draw[draw=black, fill=mypurple, thick] (5,0.4) rectangle ++(0.8,0.5);
		\draw[draw=black, fill=mypetrol,  postaction={
        pattern=north east lines, pattern color = black},thick] (5.8,0.4) rectangle ++(2.1,0.5);
		
		\draw[opacity=1,thick] (7,0) -- (7,2);
		\node at (7,2.2) {\footnotesize $\K$};
		\end{tikzpicture} 
\caption{Illustration of a sample construction of generation $n+1$ in a two-size Wright--Fisher model. The two sizes of the rectangles correspond to the two types: small rectangles represent small individuals. Each individual in generation $n$ is assigned a color and a pattern for identification. The individuals from generation $n$ are sampled, and their offspring are placed one after another from left to right until the capacity $R$ is reached. In this example three individuals from generation $n$ place exactly one offspring in generation $n+1$, one individual places two offspring and two individuals do not place any offspring. Note that the offspring of the small yellow individual mutates from small to large.} 
\label{fig.illustration}
\end{figure}

The main objective of this paper is to demonstrate that, as the resource parameter $\K$ becomes large, the appropriately scaled version of our model converges to a Wright--Fisher-type diffusion. This limiting process features a drift term that incorporates the effects of the stopping rule, frequency-dependent selection, and mutation, along with a non-standard diffusion coefficient.

\begin{theorem}\label{main}
Let $\rho:[0,1] \to \Rb$ be a Lipschitz continuous function. Suppose that $X_{0}^{\K}\to x_{0}\in[0,1]$ in probability and that $\K (\rho_{\K}(x)-x) \to \rho(x)$ uniformly in $x\in[0,1]$, as $\K\to\infty$. Then the process $\big(X_{\lfloor \K t\rfloor}^{\K}\big)_{t\geq 0}$ converges in distribution to the solution $(X_{t})_{t\geq 0}$ of the stochastic differential equation
\begin{equation} \label{eq.SDE}
\dd X_{t}\ =\ \big( - (1-\del) X_{t} (1-X_{t}) + \rho(X_{t}) \big)\,\dd t\,+\,\sqrt{X_{t}(1-X_{t})\big(1-(1-\del) X_{t}\big)}\,\dd B_{t},
\end{equation}
with initial condition $X_{0}=x_{0}$, where $B$ denotes a standard Brownian motion.
\end{theorem}

We emphasize that, unlike $\rho_{\K}$, the parameter $\del$ does not scale with $\K$. Figure~\ref{fig.trajectories} shows simulations of the finite model alongside sample paths of the limiting SDE \eqref{eq.SDE}.

\begin{remark}\label{remark_rho_decom}
The existence and uniqueness of the solution to the SDE \eqref{eq.SDE} follow from \cite[Theorem 1]{YamadaWatanabe:1971}. Moreover, since $\rho_{R}(x) \in [0,1]$ for all $x \in [0,1]$, the limiting function $\rho$ in Theorem~\ref{main} must satisfy the boundary conditions $\rho(0) \geq 0$ and $\rho(1) \leq 0$. This ensures that the solution to the SDE \eqref{eq.SDE} remains within the interval $[0,1]$ for all $t \geq 0$. Define $\beta_{0} \coloneqq \rho(0) \geq 0$, $\beta_1 \coloneqq -\rho(1) \geq 0$, and introduce the function $\sigma(x) \coloneqq \rho(x) - \beta_{0}(1 - x) + \beta_1 x$. Then $\rho$ can be decomposed as
\begin{equation} \label{eq._rho_decom}
\rho(x) = \sigma(x) + \beta_{0}(1 - x) - \beta_1 x.
\end{equation}
Since $\rho$ is Lipschitz continuous by assumption, the same holds for $\sigma$. Additionally, $\sigma$ satisfies the boundary conditions $\sigma(0) = \sigma(1) = 0$. Thus, this decomposition highlights that $\rho$ can be interpreted as comprising a frequency-dependent selection component (represented by $\sigma$) and a mutation component (captured by $\beta_{0}$ and $\beta_1$).
Furthermore, if $\rho_{R}$ admits the decomposition \eqref{sel-mut} with $\beta_{0,R}, \beta_{1,R} \in [0,1]$ and $s_{R} : [0,1] \to [0,1]$ satisfying $s_{R}(0) = 0$ and $s_{R}(1) = 1$, then the uniform convergence of $R(\rho_{R}(x) - x)$ to $\rho(x)$ as $R \to \infty$ implies that
$$ R\, \beta_{i,R}\xrightarrow[R\to\infty]{}\beta_i,\,i\in\{0,1\},\quad\text{and}\quad \sup_{x\in[0,1]}\lvert R\big(s_{R}(x)-x\big)-\sigma(x)\rvert\xrightarrow[R\to\infty]{}0. $$
\end{remark}

\begin{figure}[h] 
\begin{subfigure}[c]{0.49\textwidth}
\includegraphics[width=0.99\textwidth]{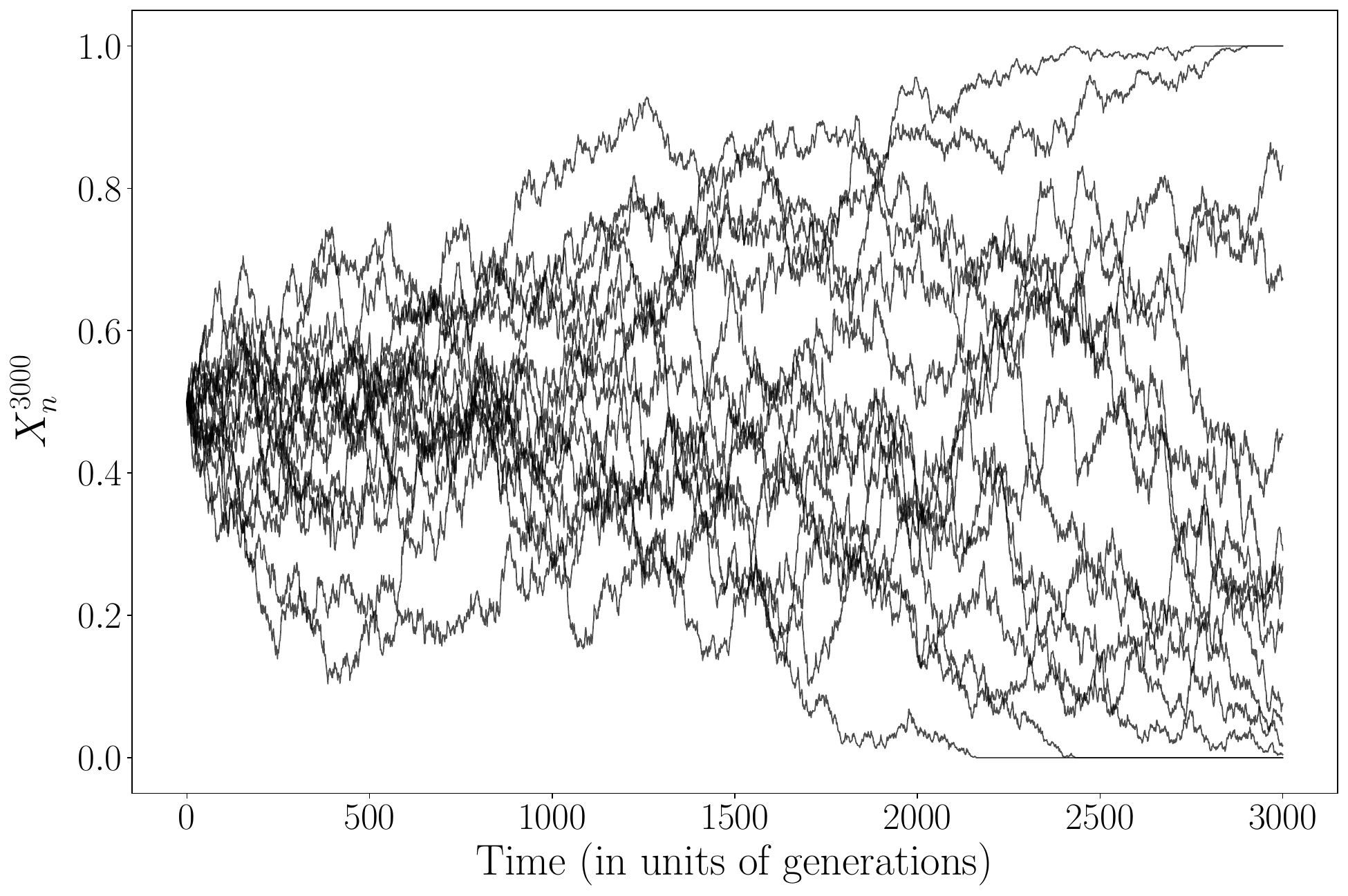} 
\label{subfig.left}
\end{subfigure}
\begin{subfigure}[c]{0.49\textwidth}
\includegraphics[width=0.99\textwidth]{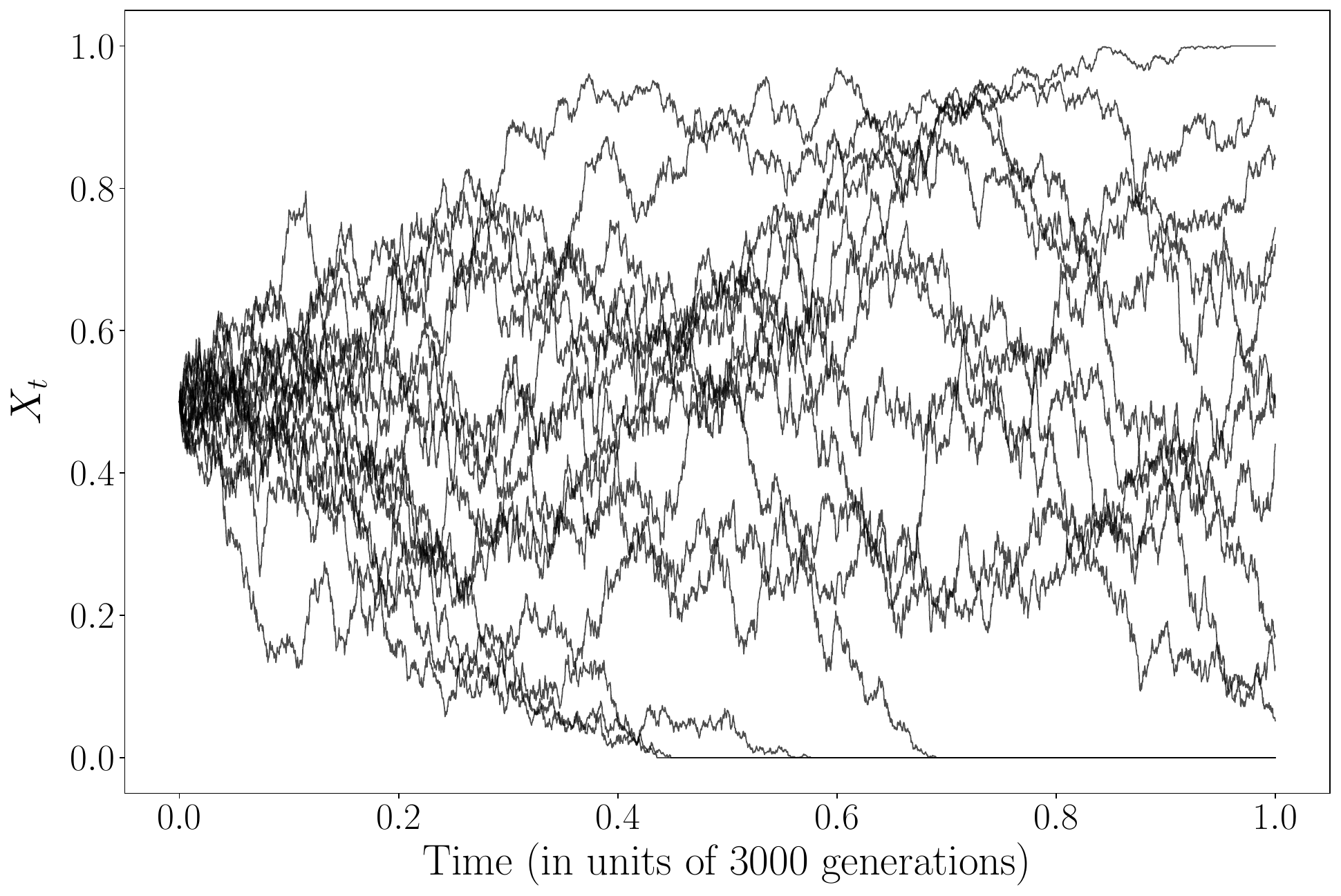} 
\label{subfig.right}
\end{subfigure}
\vspace{-1\baselineskip}
\caption{Simulations of the evolution of the proportion of small individuals in the finite model (\textit{left}) with $\rho_{\K}(x) = x$, and sample trajectories of the limiting SDE (\textit{right}) with $\rho \equiv 0$. In both cases, the parameters are $\del = 0.6$ and $x_{0} = 0.5$. The finite model was simulated for 3000 generations with $R = 3000$. SDE~trajectories were generated using the Euler method with step size $h = 1/3000$.} 
\label{fig.trajectories}
\end{figure}

\begin{remark}\label{rem:main}
Theorem~\ref{main} extends Theorem 1 in \cite{GMP20}, which considers the special case of genic selection favoring type-1 individuals without mutation, i.e.
$$ s_{R}(x)\,=\,\frac{(1-sR^{-1})x}{1-sR^{-1}x}\quad \text{and} \quad \beta_{0,R}\,=\,\beta_{1,R}\,=\,0, $$
so that $\rho_{R}(x) = s_{R}(x)$ and $\rho(x) = -s x(1 - x)$. In \cite{GMP20}, this setup is referred to as the \emph{Wright--Fisher model with efficiency}, where "efficient" denotes the small individuals. We avoid the term “efficiency” here, as it indicates an inherent advantage for small individuals -- an interpretation not supported by our findings.
Theorem~\ref{main} reveals that the drift term derived in \cite{GMP20} is in fact incorrect. That work claims~that the consumption strategy -- represented by the parameter $\vartheta$ -- has no effect on the drift term in the diffusion limit, which is asserted to coincide with the drift in the classical Wright--Fisher diffusion with genic selection. However, our results show that the drift consists of two components:
\begin{enumerate}[left=1pt]
\item A stopping bias term, explicitly depending on the size parameter $\vartheta$, which encodes a disadvantage for small individuals introduced by the stopping rule. This effect is analogous to the waiting-time paradox, which in our context implies that the event that the last individual in a generation is large has at least probability $1 - \rho_{R}(x)$; see Subsection~\ref{subsec:stop summand}. The magnitude of this disadvantage scales with $1 - \vartheta$, meaning that smaller values of $\vartheta$ intensify the effect.
\item A selection term, accounting for the sampling probabilities $\rho_{\K}$, which reduces to the drift term obtained in \cite[Theorem 1]{GMP20} in their particular setting.
\end{enumerate}
The incorrect conclusion in \cite{GMP20} stems from an assumption of exchangeability, a standard property in classical population genetics models. However, this assumption fails for the two-size Wright--Fisher model, as the stopping rule introduces a structural bias favoring large individuals. This breakdown of exchangeability will become evident in our renewal-theoretic analysis.

\smallskip 

To corroborate our theoretical results, we performed simulations whose outcomes are displayed in Figure~\ref{fig.drift_E1}. As predicted by Theorem~\ref{main}, the quantity $R\,\Eb_{x}[X_{1}^{R} - x]$ from the finite model also approximates the drift term in the limiting SDE \eqref{eq.SDE}; see also \eqref{eq.gen1}.

\begin{figure}[h!]
\begin{subfigure} [c]{0.49\textwidth}
\includegraphics[width=0.99\textwidth]{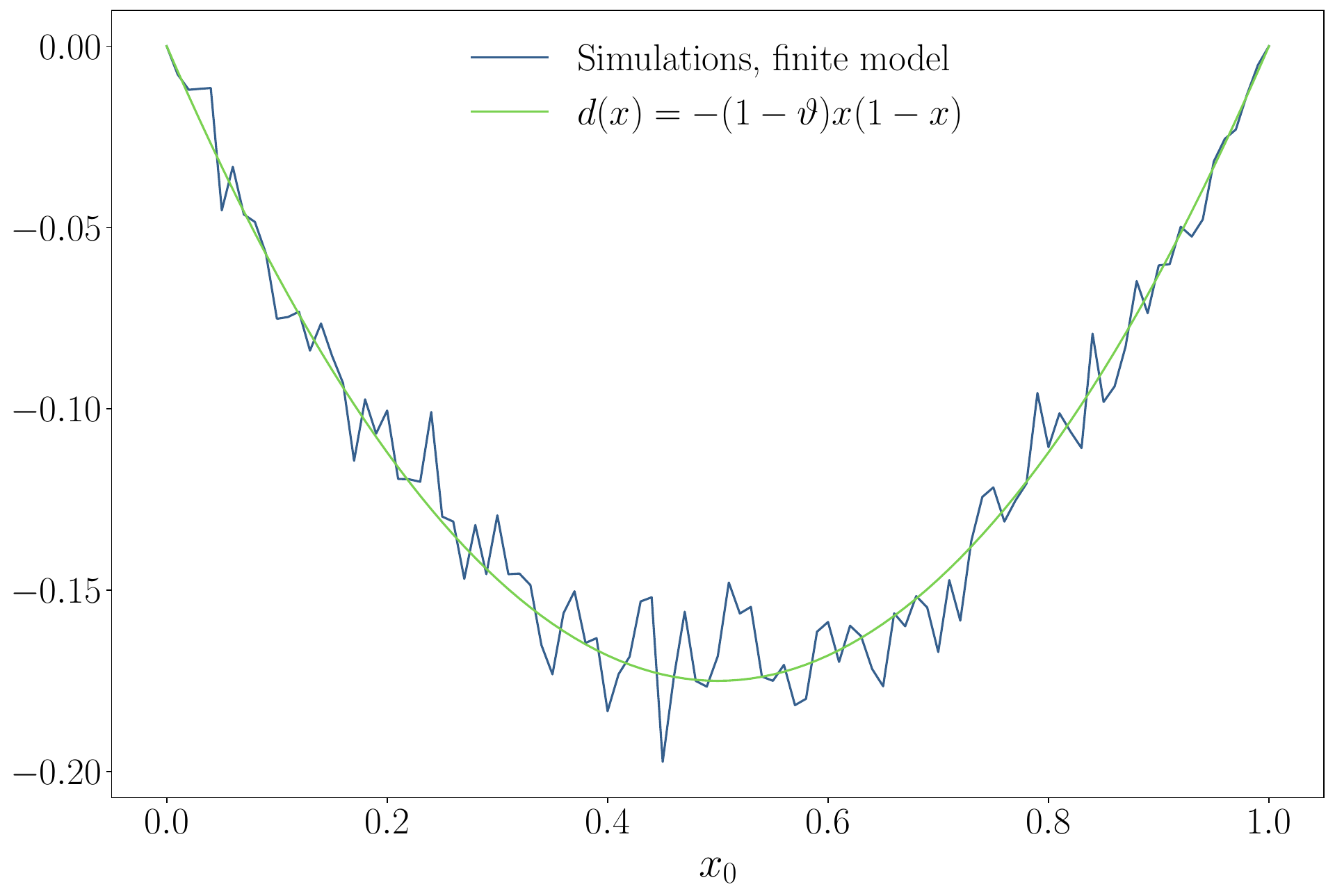} 
\end{subfigure}
\begin{subfigure}[c]{0.49\textwidth}
\includegraphics[width=0.99\textwidth]{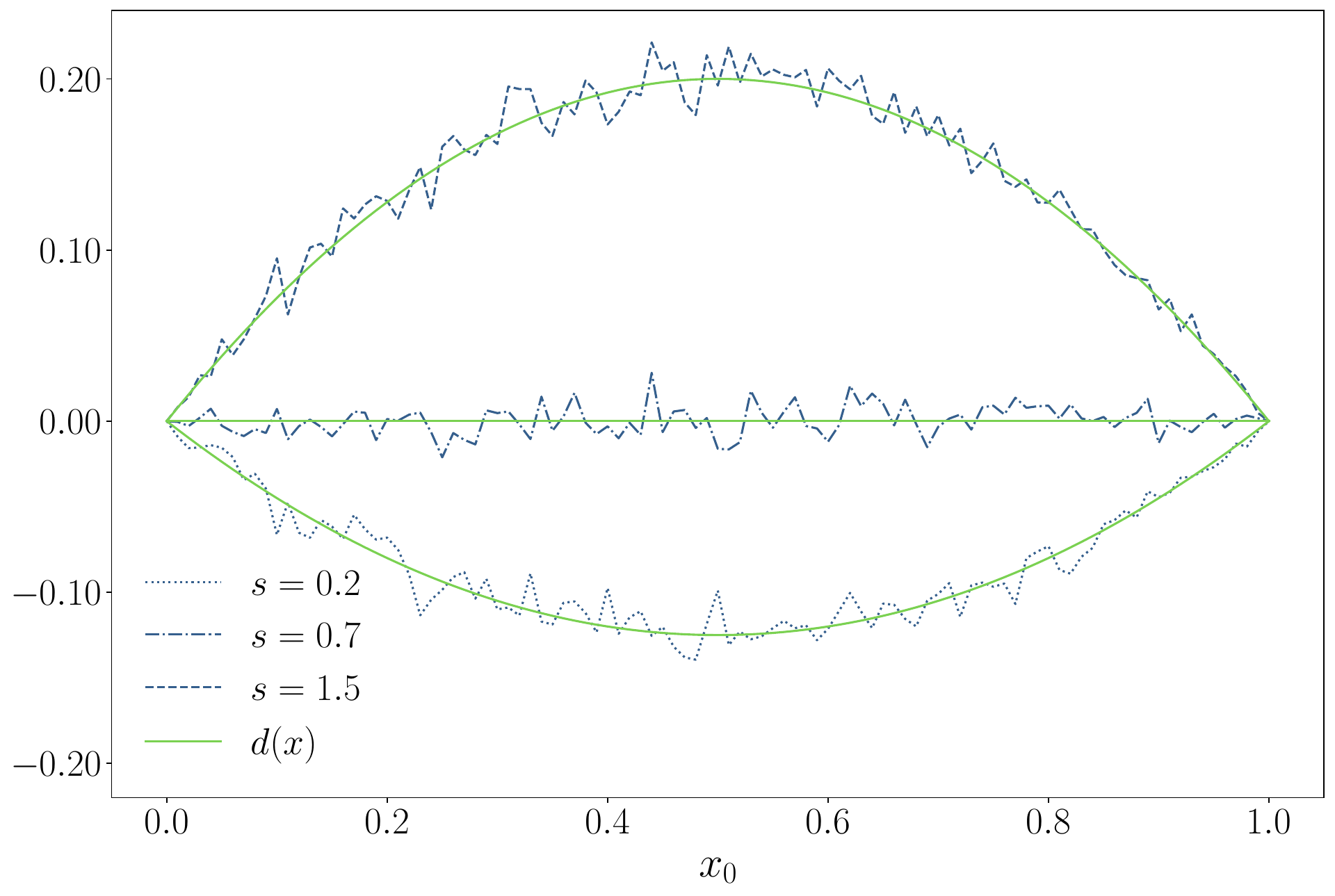}
\end{subfigure}
\caption{Approximation (blue) of $R\,\Eb_{x}[X_{1}^{R} - x_{0}]$ in the two-size Wright--Fisher model without selection (\textit{left}) and with genic selection favoring small individuals (\textit{right}). For each $x_{0} = i/100$, $i \in \{0,\dots,100\}$, we approximated $R\, \Eb_{x}[X_{1}^{R} - x_{0}]$ by simulating $R(X_{1}^{R} - x_{0})$ $10^6$ times and computing the mean. The small individuals' size parameter was set to $\del = 0.3$, and the resource capacity was $R = 1000$. The theoretical drift term $d(x_{0}) = (-(1 - \vartheta) + s)x_{0}(1 - x_{0})$ from the SDE \eqref{eq.SDE} is plotted in green.}\label{fig.drift_E1}
\end{figure}
\end{remark}

\begin{remark}\label{remark gillespie}
Although the diffusion coefficient in \eqref{eq.SDE} differs from the classical Wright--Fisher form $\sqrt{x(1 - x)}$, it is not new in the population genetics literature. For instance, Gillespie \cite{Gillespie:1974} derived the stochastic differential equation
\begin{equation*}
\dd \widetilde{X}_{t}\ =\ \big((\sigma_1^2 - \sigma_{0}^2) + (\mu_{0} - \mu_1)\big)\widetilde{X}_{t}(1 - \widetilde{X}_{t})\,\dd t + \sqrt{\widetilde{X}_{t}(1 - \widetilde{X}_{t})\big(\sigma_{0}^2 \widetilde{X}_{t} + \sigma_1^2(1 - \widetilde{X}_{t})\big)}\,\dd B_{t}
\end{equation*}
as an approximation to the type composition in a discrete-time population model with two types of individuals. For comparison with our model, we refer to them as type-$0$ and type-$1$. The two types differ in the mean and variance of their offspring numbers: the mean number of offspring for type-$i$ individuals is $1 + \mu_i$, and the variance is $\sigma_i^2$. In this formulation, the limiting process $\widetilde{X}_{t}$ tracks the proportion of type-$0$ individuals.
Specializing to the case $\sigma_{0}^2 = \del$, $\sigma_1^2 = 1$, and $\mu_{0} = \mu_1$, Gillespie’s diffusion reduces to
\begin{equation*}
\dd \widetilde{X}_{t}\ =\ (1 - \del)\widetilde{X}_{t}(1 - \widetilde{X}_{t})\,\dd t + \sqrt{\widetilde{X}_{t}(1 - \widetilde{X}_{t})\big(1 - (1 - \del)\widetilde{X}_{t}\big)}\,\dd B_{t},
\end{equation*}
which differs from our SDE \eqref{eq.SDE} with $\rho(x) = 0$ only in the sign of the drift term. In our model, type-$0$ individuals (small-sized) are at a disadvantage due to the stopping rule, whereas in Gillespie's model, type-$0$ individuals (those with lower offspring variance) are favored. One of the central conclusions in \cite{Gillespie:1974} is that reduced variance in offspring number can confer a selective advantage.

Despite the distinct modeling assumptions, the agreement in the diffusion terms is not coincidental. It reflects the fact that differences in size (in our model) and differences in offspring variance (in Gillespie’s model) can lead to the same effective population size. This shared feature explains the identical form of the diffusion coefficient in both settings; see Remark~\ref{remark.eff_pop_size} and \cite[p.~605]{Gillespie:1974}.
\end{remark}

\begin{remark} \label{special cases}
We now provide intuition for the sampling and reproduction mechanisms underlying our model. In general, selection governs how parents are sampled, while mutation determines how offspring are produced. The combined effect of these evolutionary forces is encoded in the function $\rho_{R}$, which specifies the probability that a sampled individual produces a small offspring.

Below, we present several common scenarios covered by our framework and state the corresponding functions $\rho_{R}$ (from the finite model) and $\rho$ (from the drift term in \eqref{eq.SDE}):

\begin{enumerate}[left=4pt]\itemsep2pt
\item No selection, no mutation:
    \begin{equation*}
    s_{R}(x)\,=\,x \quad\text{and}\quad\beta_{0,R}\,=\,\beta_{1,R}\,=\,0,
    \end{equation*}
    so that $\rho_{R}(x)=x$ and $\rho(x)=0$.
\item Genic selection (favoring small individuals): 
\begin{align*}
s_{R}(x)\,=\, \frac{(1+s\K^{-1})x}{1+s\K^{-1}x}  \quad\text{and}\quad\beta_{0,R}\,=\,\beta_{1,R}\,=\,0.
\end{align*}
Consequently $\rho_{R}(x)=s_{R}(x)$ and $\rho(x)=s x(1-x)$, with $s\geq0$. The adaptation to genic selection favoring large individuals is straightforward and yields $\rho(x)=-sx(1-x)$, again with $s\geq 0$, see Theorem \ref{main} and compare with \cite[Theorem 1]{GMP20}.
\item Fittest-type-wins selection (favoring small individuals): 
$$ s_{R}(x)=\,1-\Eb\big[ (1-x)^G \big] \quad\text{and}\quad\beta_{0,R}\,=\,\beta_{1,R}\,=\,0, $$
where $G$ is a $\Nb$-valued random variable with
\begin{equation*}
\Pb(G=1)\,=\,1-\frac{1}{\K} \quad \text{and}\quad\Pb(G=k)\,=\,\frac{s_{k-1}}{\K}\text{ for }k\geq 2,
 \end{equation*}
and weights $s_k \geq 0$ satisfying $\sum_{k = 1}^{\infty} s_k = 1$. In this case, $\rho_{R}(x) = s_{R}(x)$ and the limiting function is
\begin{equation*}
\rho(x)\,=\,s(x)x(1-x) \quad \text{ with } s(x)\,=\,\sum_{k=1}^{\infty} s_{k} (1-x)^{k}.
\end{equation*}
The variable $G$ can be interpreted as the number of "potential parents" in the underlying ancestral picture (see \cite{BEH23}, \cite{GS17}).
\item Diploid selection:
$$s_{R}(x)\,=\,x+\frac{2s}{\K}\,x(1-x)\big((1-2h)x+h\big) \quad\text{and}\quad\beta_{0,R}\,=\,\beta_{1,R}\,=\,0, $$ 
with $s \geq 0$ and $h \in \Rb_+$. Then $\rho_{R}(x) = s_{R}(x)$ and
$$ \rho(x)\,=\,2 s\, x(1-x)\big((1-2h)x+h\big). $$
In the standard Wright--Fisher model the haploid population can also be interpreted as a diploid setting where each genotype $ij \in\{0,1\}^2$ has a fitness value $w_{ij}$, see \cite[Chapter 5]{Ewens:04}. The homozygote $00$ (resp. $11$) reproduces with rate $w_{00} = 1 + 2s$ (resp. $w_{11} = 1$) and the heterozygots have rate $w_{01} = w_{10} = 1 + 2hs$. The parameter $s \geq 0$ controls the strength of selection, and $h$ (the dominance parameter) measures the contribution of allele $0$ to the fitness of a heterozygote. The case $h = 1/2$ corresponds to additive selection (no dominance), $h < 1/2$ models the case where allele $0$ is recessive, $h > 1/2$ represents a setting where allele $0$ is dominant, and $h > 1$ corresponds to balancing selection. Even in the absence of a direct diploid interpretation, the functions $\rho_{R}$ and $\rho$ serve to model various diploid selection regimes in the two-size Wright--Fisher framework.
\item Parent-independent mutation, no selection:
$$ s_{R}(x)\,=\,x \quad \text{and}\quad \beta_{i,R}\,=\,\frac{\beta_i}{R}\,\geq 
\,0. $$ 
Then,
\begin{equation*}
\rho_{\K}(x)\,=\,x\Big(1-\frac{\beta_{1}}{R}\Big)+(1-x)\frac{\beta_{0}}{R}\quad\text{and}\quad \rho(x)\,=\,\beta_{0}(1-x)-\beta_1x.
\end{equation*}
\end{enumerate}
\end{remark}

We will revisit the cases of genic selection and parent-independent mutation in Section~\ref{sec:asympt}, where we derive asymptotic properties of the corresponding two-size Wright--Fisher diffusions using classical diffusion theory.

\subsection{A Model Variant}\label{subsec:model variant}

A natural variant of our model arises by slightly altering the stopping rule: instead of completing a generation with the first individual that causes the total resource consumption to exceed $\K$, one may instead reject this individual and end the generation at the previous one (with the outcome unchanged only when the total exactly equals $\K$). Let $\bX_{n}^R$ denote the corresponding process in this variant. A similar analysis to that of the original model yields the following analogue of Theorem~\ref{main}.

\begin{theorem}\label{thm:model variant}
Under the same assumptions as in Theorem~\ref{main}, the process $\big(\bX_{\lfloor \K t \rfloor}^\K\big)_{t \geq 0}$ converges in distribution to the solution $(\bX_{t})_{t \geq 0}$ of the stochastic differential equation
\begin{equation*} 
\dd \bX_{t}\ =\ \rho(\bX_{t})\,\dd t + \sqrt{\bX_{t}(1-\bX_{t})\big(1-(1-\del) \bX_{t}\big)}\,\dd B_{t} ,
\end{equation*}
with initial condition $\bX_{0} = x_{0}$, where $B$ denotes a standard Brownian motion.
\end{theorem}

The key difference from the SDE in \eqref{eq.SDE} lies in the drift term: for $\rho(x) = 0$, the original model has a drift of $-(1 - \del)x(1 - x)$, whereas the variant has zero drift and is thus neutral. We refer to Section~\ref{sec:variant} for further details and a proof sketch, which closely parallels the argument used for Theorem~\ref{main}.

\begin{remark}
This variant (with $\vartheta \in \Qb$) was also studied in \cite{GMP20}, again in the special case $\rho_\K(x) = \frac{(1 - s \K^{-1}) x}{1 - s \K^{-1} x}$. While we show that the consumption strategy has no impact on the drift in this variant, Theorem 2 of \cite{GMP20} incorrectly suggests a selective advantage for small individuals, and their drift term varies significantly with different values of $\vartheta \in \Qb$. This error stems from the same incorrect assumption discussed for the original model. Our simulations, shown in Figure \ref{fig.E2}, support the theoretical findings presented here.
\end{remark}

\begin{figure}[h!]
\begin{subfigure} [c]{0.50\textwidth}
\includegraphics[width=0.99\textwidth]{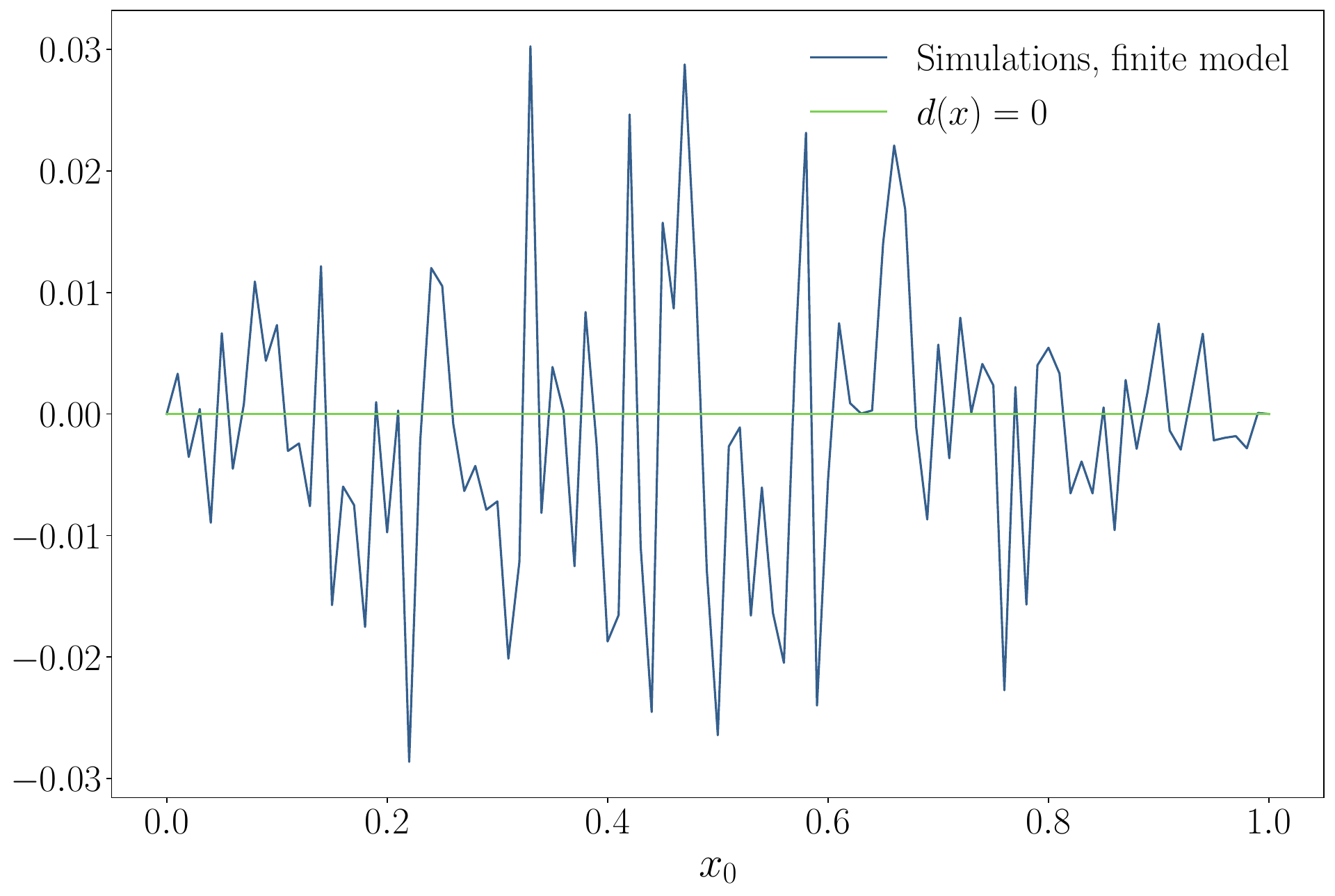}
\caption{$R \, \Eb_{x}[\bX_{1}^{R}-x_{0}]$, with $n_{\text{sim}}=10^6$} \label{driftE2}
\end{subfigure}
\begin{subfigure}[c]{0.49\textwidth}
\includegraphics[width=0.99\textwidth]{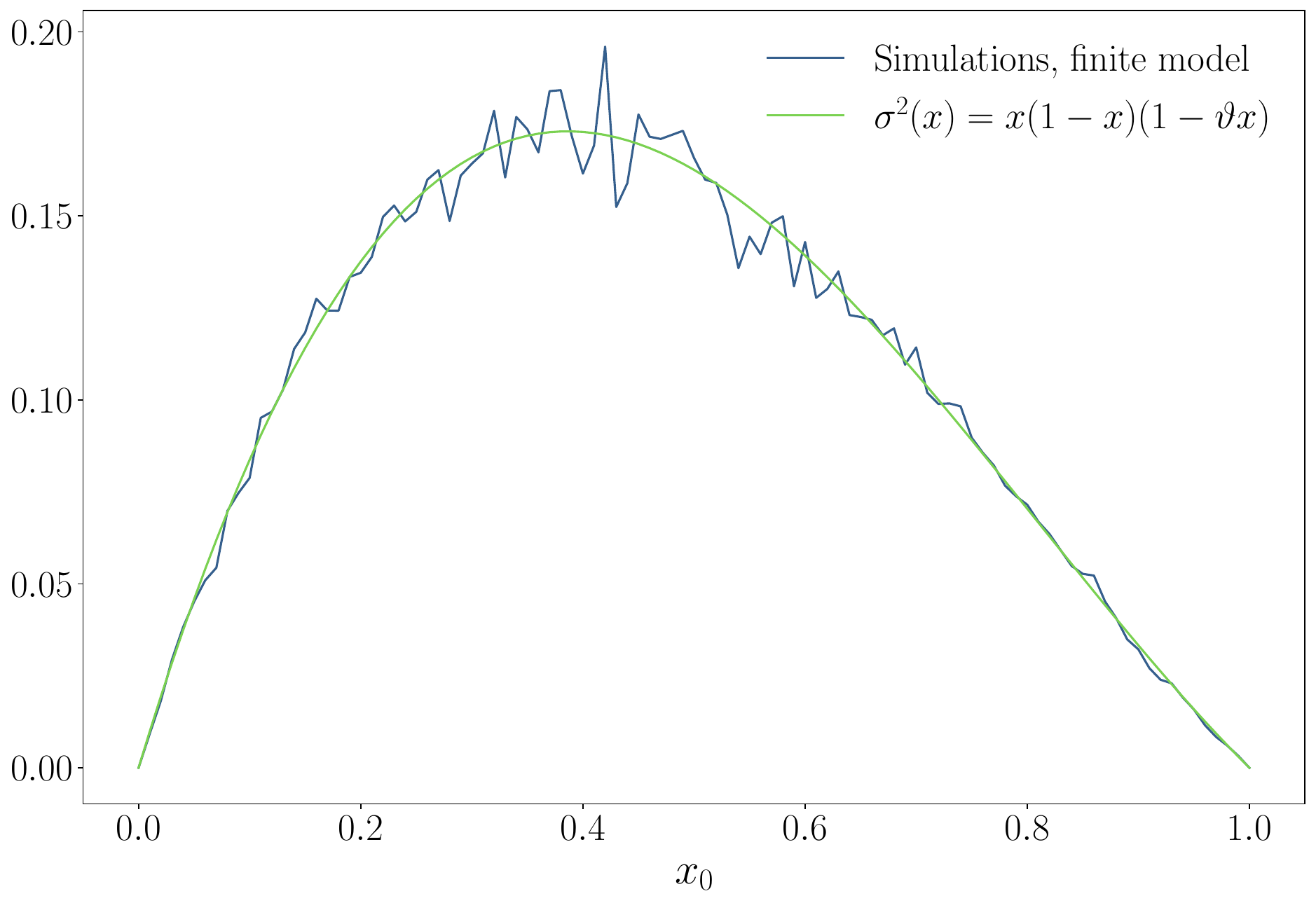}
\caption{$R \, \Eb_{x}[(\bX_{1}^{R}-x_{0})^2]$, with $n_{\text{sim}}=10^3$} \label{diffusionE2}
\end{subfigure}
\caption{Approximations (blue) of $R, \Eb_{x}[\bX_{1}^{R} - x_{0}]$ (\textit{left}) and $R, \Eb_{x}[(\bX_{1}^{R} - x_{0})^2]$ (\textit{right}) for the variant of the two-size Wright--Fisher model with $\rho_{R}(x) = x$. For each $x_{0} = i/100$, $i \in \{0, \dots, 100\}$, we estimated the expectations by simulating $R(\bX_{1}^{R} - x_{0})$ and $R(\bX_{1}^{R} - x_{0})^2$ over $n_{\text{sim}}$ runs and computing the sample means. The size of small individuals was set to $\del = 0.3$, and the resource capacity to $R = 1000$. The drift function $d(x)$ (left) and the diffusion coefficient $\sigma^2(x)$ (right) from the SDE in Theorem~\ref{thm:model variant} are shown in green.}
\label{fig.E2}
\end{figure}

\subsection{Connection to renewal theory}\label{subsec:connection}

As already indicated, a key ingredient in our analysis is the connection between the one-step transitions of the process $\big(X_{n}^{\K}, M_{n}^{\K}\big)_{n \geq 0}$ and classical renewal theory. This connection is formalized via the distributional identity in \eqref{l1} below. We fix $\del \in (0,1)$, and let $(\Omega, \mathcal{F}, \Pb)$ be a probability space that supports both the process $\big(X_{n}^{\K},M_{n}^{\K}\big)_{n \geq 0}$ and a sequence $(\xi_i)_{i \geq 1}$ of $\{\del, 1\}$-valued random variables satisfying the following conditions:

\begin{enumerate}[left=4pt]\itemsep2pt
\item The $\xi_i$ are iid under each $\Pb_{x} := \Pb(\cdot \mid X_{0}^{\K} = x)$, $x\in [0,1]$, with distribution $F_{\rho_{\K}(x)}$ and mean $\Eb_{x}[\xi_i]=\mu(\rho_{\K}(x))$, where
\begin{equation*}
F_{p}\,\coloneqq\,p\, \delta_{\del}+(1-p)\, \delta_{1}\quad\text{and}\quad\mu(p)\,\coloneqq\,\int u\,F_{p}(\dd u)\,=\,1-(1-\del)p\quad\text{for }p\in [0,1],
\end{equation*}
and $\delta_{x}$ denotes the Dirac measure at $x$.

\item The sequences $(\xi_i)_{i \geq 1}$ and $\big(X_{n}^{\K}, M_{n}^{\K}\big)_{n \geq 0}$ are independent under each $\Pb_{x}$.
\end{enumerate}

To state uniform renewal results later, we also introduce a family of auxiliary probability measures $(\Pbf_p)_{p \in [0,1]}$ on $(\Omega, \mathcal{F})$, under which the $(\xi_i)_{i \geq 1}$ are iid with law $F_p$ and mean $\mu(p)$. When $p = \rho_{\K}(x)$ for some $x \in [0,1]$, we can take $\Pbf_p = \Pb_{x}$. Expectations under $\Pb_{x}$ and $\Pbf_p$ are denoted by $\Eb_{x}$ and $\Ebf_{p}$, respectively.

Now define the zero-delayed renewal process $S = (S_{n})_{n \geq 0}$ by
\begin{equation*}
S_{0}\,:=\,0, \qquad S_{n}\,:=\,\sum_{j=1}^n \xi_j \quad \text{for } n \geq 1,
\end{equation*}
and its first passage time above level $a \geq 0$ by
\begin{equation}\label{eq.def_stop}
\tau(a)\,:=\,\inf\{n \in \Nb : S_{n} \geq a\}.
\end{equation}

We now relate this renewal process to the one-step transitions of our model. Let $S_i^{\K}$ denote the total resources consumed to produce the first $i$ individuals in generation 1. The total number of individuals in this generation is then given by
$$ M_{1}^{\K}\ =\ \inf\{n\in\Nb:{S}_{n}^{\K}\geq \K\}. $$
Since individuals are either of size $\del$ or 1, the total resources required for generation 1 satisfy
$$ {S}_{M_{1}^{\K}}^{\K}\ =\ \del M_{1}^{\K} X_{1}^{\K}\,+\,M_{1}^{\K} (1-X_{1}^{\K})\ =\ -(1-\del) M_1^{\K}X_{1}^{\K}\,+\,M_{1}^{\K} \quad\Pb_{x}\text{-a.s.} $$
It follows directly from the model that the vectors $(S_1^{\K}, \ldots, S_{M_1^{\K}}^{\K})$ and $(S_1, \ldots, S_{\tau(\K)})$ have the same law under $\Pb_{x}$. Therefore,
\begin{equation}\label{l1}
\Pb_{x}\big((X_1^{\K}, M_1^{\K}) \in \cdot\,\big)\ =\ \Pb_{x}\Bigg(\bigg(\frac{1}{1 - \del} \Big(1 - \frac{S_{\tau(\K)}}{\tau(\K)}\Big), \tau(\K)\bigg) \in \cdot\,\Bigg),
\end{equation}
i.e., the one-step dynamics of the two-size Wright--Fisher model are fully determined by the renewal process $S$ and its stopping time $\tau(\K)$. In particular, we have the identity
\begin{equation}\label{l1a}
X_1^{\K} - \rho_{\K}(x)\ =\ -\frac{1}{1 - \del} \Big( \frac{S_{\tau(\K)}}{\tau(\K)} - \mu\big(\rho_{\K}(x)\big) \Big) \quad \Pb_{x}\text{-a.s.}
\end{equation}

\subsection{Proof strategy via uniform renewal theory}\label{subsec:proof strategy via uniform}

To prove Theorem \ref{main}, we will show that for any $f \in C^4([0,1])$, the discrete generator
\begin{equation*}
\mathcal{A}^{\K} f(x)\,\coloneqq\,\K \, \Eb_{x}\big[f(X_{1}^{\K}) - f(x)\big]
\end{equation*}
converges uniformly in $x$ to the infinitesimal generator $\mathcal{A}f(x)$ of the limiting diffusion process defined by the SDE \eqref{eq.SDE}, as $\K \to \infty$. The conclusion then follows by a classical convergence result for Markov processes, see Ethier and Kurtz \cite[Theorem 1.6.1]{EK86} or Kallenberg \cite[Theorem 17.25]{Kallenberg:21}.

\vspace{.1cm}
By Itô's formula, the generator $\mathcal{A}$ acts on functions $f \in C^2([0,1])$ as
\begin{equation} \label{eq.gen_SDE}
\mathcal{A}f(x) = \big(-(1 - \del)x(1 - x) + \rho(x)\big) f'(x) + \frac{1}{2}x(1 - x)\big(1 - (1 - \del)x\big) f''(x).
\end{equation}
To relate this to the discrete generator, we perform a fourth-order Taylor expansion, yielding
\begin{align*}
\mathcal{A}^{\K} f(x)\ =&\ \K \, \Eb_{x}\big[ X_{1}^{\K} - x \big] f'(x) + \frac{1}{2} \K \, \Eb_{x}\big[ (X_{1}^{\K} - x)^2 \big] f''(x) + \frac{1}{6} \K \, \Eb_{x}\big[ (X_{1}^{\K} - x)^3 \big] f'''(x) \\
&+ \frac{1}{12} \K \, \Eb_{x}\big[ (X_{1}^{\K} - x)^4 f^{(4)}(Z_{x}) \big],
\end{align*}
for some random point $Z_{x}$ in $[0,1]$. Using the bound
\begin{equation*}
  \big\vert \Eb_{x}\big[ (X_{1}^{\K}-x)^4 f^{(4)}(Z_{x})\big]\big\vert\ \leq\  \Eb_{x}\big[ (X_{1}^{\K}-x)^4\big] \, \Vert f^{(4)}\Vert_{\infty},
\end{equation*}
where $\Vert\cdot\Vert_{\infty}$ denotes the uniform norm, it suffices to prove that
\begin{align}
\K\,\Eb_{x}[X_{1}^{\K}-x]\,&\xrightarrow[\K\to\infty]{}\, -(1-\del) x (1-x)+\rho(x),\label{eq.gen1}\\
\K\,\Eb_{x}\big[(X_{1}^{\K}-x)^2\big]\,&\xrightarrow[\K\to\infty]{}\, x(1-x)\big(1-(1-\vartheta) x\big),\label{eq.gen2}\\
\K\,\Eb_{x}\big[ (X_{1}^{\K}-x)^3\big]\,&\xrightarrow[\K\to\infty]{}\,0,\label{eq.gen3}\\
\K\,\Eb_{x}\big[ (X_{1}^{\K}-x)^4\big]\,&\xrightarrow[\K\to\infty]{}\,0,\label{eq.gen4}
\end{align}
uniformly in $x\in[0,1]$.

\vspace{.1cm}
These convergence statements will be established in Section \ref{sec:proof main}, after preparing the necessary tools from renewal theory. The connection to the latter becomes evident by observing that, via identity \eqref{l1a}, the centered moments $\Eb_{x}[(X_{1}^{\K} - \rho_{\K}(x))^n]$ for $n = 1,2,3,4$ can be written in terms of $S_{\tau(\K)}$ and $\tau(\K)$
\begin{equation}\label{eq:moments->renewal theory}
\Eb_{x}\big[\big(X_{1}^{\K} - \rho_{\K}(x)\big)^n\big]\ =\ \frac{(-1)^n}{(1 - \del)^n} \Eb_{x}\left[ \Big( \frac{S_{\tau(\K)}}{\tau(\K)} - \mu\big(\rho_{\K}(x)\big) \Big)^n \right].
\end{equation}
Two key ingredients in proving the uniform convergence of the generator are:
\begin{enumerate}[left=4pt]\itemsep2pt
\item a uniform version of an $L^p$-type elementary renewal theorem (see \eqref{eq:ERT L_p version} below), and
\item a uniform weak convergence result for the stopping summand $\xi_{\tau(\K)}$.
\end{enumerate}
In our setting, the classical (pointwise) elementary renewal theorem states that, for each $p \in [0,1]$,
\begin{equation*}
\lim_{\K\to\infty}\frac{\K}{\tau(\K)}\ =\ \mu(p) \quad\Pbf_{p}\text{-almost surely},
\end{equation*}
see e.g.\ \cite[Theorem 2.5.1 and Remark 2.5.1]{Gut:09}. Moreover, since $\del \le \xi_j \le 1$ for all $j$, we obtain the uniform bounds
\begin{equation}\label{eq:stopping bounds}
\K\,\le\,\tau(\K)\,\le\,\frac{\K+1}{\del}\quad\text{for all }\K>0
\end{equation}
which, combined with dominated convergence, yield the following $L^\beta$ version
\begin{equation}\label{eq:ERT L_p version}
\lim_{\K\to\infty}\Ebf_{p}\Bigg[\bigg(\frac{\K}{\tau(\K)}\bigg)^{\beta}\Bigg]\,=\,\mu(p)^{\beta}
\quad\text{for all }\beta>0\text{ and }p\in [0,1].
\end{equation}
Additionally, it is well-known (see \cite[Theorems 2.10.2 and 2.10.3]{Thorisson:00}) that the law $Q_{p}^{\K}$ of the stopping summand $\xi_{\tau(\K)}$ under $\Pbf_p$ converges weakly to a limiting law $Q_{p}$, which can be identified by solving a renewal equation and is given by
\begin{equation}\label{eq:form of Qp}
Q_{p}\,=\,\frac{\del \, p}{\mu(p)}\delta_{\del}\ +\ \frac{1-p}{\mu(p)}\delta_{1}.
\end{equation}
This is immediate for $p = 0$ and $p = 1$, and follows for $p \in (0,1)$ by a standard coupling argument. Since the support of the $Q_{p}^{\K}$ is a two-point set, the weak convergence may be equivalently written as
\begin{equation}\label{eq:QpR->Qp}
\lim_{\K \to \infty} \big| Q_{p}^{\K}({\del}) - Q_{p}({\del}) \big|\ =\ 0.
\end{equation}
No lattice-type considerations are needed because the support of $\xi_{\tau(\K)}$ does not vary with $\K$ -- unlike the support of the excess over the boundary $S_{\tau(\K)}-\K$ in the case when $\del\in\Qb$ and thus $(S_{n})_{n\ge 0}$ is arithmetic.

\vspace{.1cm}
In the next section, we will prove that this convergence holds uniformly in $p \in [0,1]$, for fixed $\del \in (0,1)$ (see Proposition \ref{uniform stopping summand}). This will allow us to establish a uniform extension of \eqref{eq:ERT L_p version} involving the stopping summand (Proposition \ref{prop:uniform key renewal}). These results together will imply uniform convergence of the centered moments of $S_{\tau(\K)}/\tau(\K)$ under $\Pbf_p$ for $p \in [0,1]$, which, through identity \eqref{eq:moments->renewal theory}, will yield \eqref{eq.gen1}\,--\,\eqref{eq.gen4} and thus complete the proof of Theorem \ref{main} in Section \ref{sec:proof main}.

\section{Uniform Renewal Theorems}\label{sec:URT}

\subsection{The Stopping Summand}\label{subsec:stop summand}

Let $Q_{p}^{\K}$ and $Q_{p}$ be as introduced previously, and denote by $\xi_{\infty}$ a random variable with law $Q_{p}$ under $\Pbf_{p}$, independent of all other relevant random variables. Let $\xi$ be a random variable with distribution $F_p$ under $\Pbf_{p}$ and independent of the $\xi_i$. Note that for all $p \in [0,1]$,
$$ \Ebf_{p}[\xi_{\infty}]\ =\ \frac{1-p(1-\del)(1+\del)}{\mu(p)}\ =\ \frac{\Ebf_{p}[\xi^{2}]}{\mu(p)}. $$

\begin{proposition}\label{uniform stopping summand}
Fix $\del \in (0,1)$. Let $S$ be a renewal process with stopping time $\tau(\K)$ as defined in \eqref{eq.def_stop}, and let $\xi_{\tau(\K)}$ denote the corresponding stopping summand. Then $\xi_{\tau(\K)}$ converges in~dis\-tribution as $\K \to \infty$, uniformly in $p \in [0,1]$. Specifically,
\begin{equation}\label{eq1:uniform stopping summand}
\lim_{\K \to \infty}\ \sup_{p \in [0,1]}\left|\Pbf_{p}(\xi_{\tau(\K)} = \del) - \frac{p\del}{\mu(p)}\right|\ =\ 0.
\end{equation}
As a direct consequence, for all $\beta > 0$,
\begin{equation}\label{eq2:uniform stopping summand}
\lim_{\K \to \infty}\ \sup_{p \in [0,1]} \left| \Ebf_{p}\big[\xi_{\tau(\K)}^\beta\big] - \Ebf_{p}\big[\xi_{\infty}^\beta\big] \right|\ =\ 0.
\end{equation}
\end{proposition}

Our proof of \eqref{eq1:uniform stopping summand} is purely probabilistic and based on a coupling argument. Fix $\nu \in (0, \tfrac{1}{2})$, we construct a coupling process whose distribution is the same under any $\Pbf_{p}$ with $p \in [\nu, 1 - \nu]$, implying that the coupling time has the same distribution across this range. This yields uniform convergence on $[\nu, 1 - \nu]$. Moreover, as shown in Lemma \ref{lem:aux 1} and supported by a simple intuitive argument, if $\nu$ is chosen sufficiently small, then for $p \in (0, \nu)$, the distribution $\Pbf_{p}(\xi_{\tau(\K)} \in \cdot)$ is nearly $\delta_1$, i.e., the law of $\xi_{\tau(\K)}$ under $\Pbf_{0}$. Similarly, for $p \in (1 - \nu, 1)$, it is close to $\delta_{\del}$, the law of $\xi_{\tau(\K)}$ under $\Pbf_1$. These two ingredients combine to establish the uniform convergence asserted in \eqref{eq1:uniform stopping summand}.

\vspace{.1cm}
In the arithmetic case (i.e., $\del \in \Qb$), \eqref{eq1:uniform stopping summand} can also be derived from Lemma \ref{lem:aux 1} and a result by Borovkov and Foss \cite[Theorem 2.7]{BorovkovFoss:1999}, after verifying their Fourier-analytic condition: for some continuous function $\psi$ on $[0, 2\pi]$ (assuming lattice span one),
\begin{equation} \label{eq.inequ}
\Big|\Ebf_{p}[e^{i u \xi}] -1\Big|\, \geq\, \psi(u)\quad\text{for all }u\in (0,2\pi).
\end{equation}
To the best of our knowledge, the result in the non-arithmetic case is new.

\medskip
Let us now introduce the necessary notation and auxiliary results used in the proof of Proposition \ref{uniform stopping summand}, which will follow at the end of this subsection.

\medskip
Let $\Ubf_{p}\coloneqq \sum_{n=0}^{\infty}\Pbf_{p}(S_{n} \in \cdot)$ denote the renewal measure of the process $S$. The classical version of Blackwell's renewal theorem (\cite[Theorem 2.4.2]{Gut:09}) states that
\begin{gather*}
\lim_{\K \to \infty} \Ubf_{p}\big([\K-t,\K)\big)\ =\ \frac{t}{\mu(p)}\quad\text{if $S$ is non-arithmetic}
\shortintertext{and}
\lim_{n \to \infty} \Ubf_{p}\Big(\Big\{\frac{n}{b}\Big\}\Big)\ =\ \frac{1}{b\,\mu(p)}\quad \text{if $S$ is arithmetic with lattice-span }\frac{1}{b}.
\end{gather*}
Moreover, a standard renewal argument gives
\begin{gather*}
\Pbf_{p}(\xi_{\tau(R)}=\del)\ =\ \Pbf_{p}(\xi_{1}=\del)\,\Ubf_{p}\big([R-\del,R)\big)\ =\ p\,\Ubf_{p}\big([R-\del,R)\big)
\shortintertext{and}
\Pbf_{p}(\xi_{\tau(R)}=1)\ =\ \Pbf_{p}(\xi_{1}=1)\,\Ubf_{p}\big([R-1,R)\big)\ =\ (1-p)\,\Ubf_{p}\big([R-1,R)\big)
\end{gather*}
for all $p\in [0,1]$ and $R\ge 0$, which implies the identity
\begin{equation}\label{eq:ren measure id}
 p\,\Ubf_{p}\big([R-\del,R)\big)\,+\,(1-p)\,\Ubf_{p}\big([R-1,R)\big)\,=\,1.   
\end{equation}
From this, the limiting distribution $Q_{p}$ in \eqref{eq:form of Qp} follows both in the arithmetic case ($\del \in \Qb$) and the non-arithmetic case ($\del \notin \Qb$). In the arithmetic setting, note that $Q_{p}^{\K}$, the law of $\xi_{\tau(\K)}$ under $\Pbf_{p}$, remains constant between consecutive lattice points. The form of $Q_{p}$ remains valid also at the boundary values $p = 0$ and $p = 1$, where the limiting distribution coincides trivially with the increment laws $F_{0} = \delta_1$ and $F_1 = \delta_{\del}$, respectively, as already noted.

\medskip
With the help of Lemma \ref{lem:aux 1} below, we may restrict attention to $p \in [\nu, 1 - \nu]$ for any sufficiently small $\nu > 0$, reducing the uniformity claim to showing
\begin{equation}\label{eq:reduced assertion}
\lim_{\K \to \infty}\ \sup_{p \in [\nu, 1 - \nu]}\| Q_{p}^{\K} - Q_{p}\| = 0,
\end{equation}
where $\|\cdot\|$ denotes the total variation distance (normalized). Since both $Q_{p}^{\K}$ and $Q_{p}$ are supported on the two-point set $\{\del, 1\}$, this distance simplifies to
$$ \big\|Q_{p}^{\K}-Q_{p}\big\|\ =\ \big|Q_{p}^{\K}(\{\vartheta\})-Q_{p}(\{\vartheta\})\big|\ =\ \big|Q_{p}^{\K}(\{1\})-Q_{p}(\{1\})\big|. $$

\begin{lemma}\label{lem:aux 1}
Let $Q_{p}^{\K},$ and $Q_{p}$ be as above. Then
\begin{equation*}
\lim_{p(1-p)\to 0}\,\sup_{\K\geq 0}\big\|Q_{p}^{\K}-Q_{p}\big\|\ =\ 0.
\end{equation*}
\end{lemma}

\begin{proof}
By \eqref{eq:ren measure id} and the fact that $\sup_{p\in [0,1]}\sup_{\K\geq 0}\Ubf_{p}([\K-\vartheta,\K))\leq1$, we have
\begin{align*}
\lim_{p\downarrow 0}\Ubf_{p}\big([\K-1,\K)\big)\ =\ \lim_{p\downarrow 0}\frac{1-p\,\Ubf_{p}\big([\K-\vartheta,\K)\big)}{1-p}\ =\ 1
\shortintertext{and}
\lim_{p\uparrow 1}\Ubf_{p}\big([\K-\vartheta,\K)\big)\ =\ \lim_{p\uparrow 1}\frac{1-(1-p)\,\Ubf_{p}\big([\K-1,\K)\big)}{p}\ =\ 1,
\end{align*}
both uniformly in $\K\geq 0$. Using this and the explicit form of $Q_{p}^{\K}$ from the renewal representation, we obtain
\begin{gather*}
\sup_{\K}\big\|Q_{p}^{\K}-Q_{p}\big\|\ =\ \sup_{\K}\bigg|\Pbf_{p}(\xi_{\tau(\K)}=\vartheta)-\frac{p\vartheta}{\mu(p)}\bigg|\ =\ \sup_{\K}\bigg|p\,\Ubf_{p}\big([\K-\vartheta,\K)\big)-\frac{p\vartheta}{\mu(p)}\bigg|\ \xrightarrow{p\uparrow 1}\ 0
\shortintertext{and similarly}
\sup_{\K}\big\|Q_{p}^{\K}-Q_{p}\big\|\ =\ \sup_{\K}\bigg|(1-p)\,\Ubf_{p}\big([\K-1,\K)\big)-\frac{1-p}{\mu(p)}\bigg|\ \xrightarrow{p\downarrow 0}\ 0.
\end{gather*}
This completes the proof.
\end{proof}

In the following, we restrict attention to the case $\vartheta\notin\Qb$ so that the renewal process $S$ is non-arithmetic under each $\Pbf_{p}$, $p\in (0,1)$. The arguments in the arithmetic case are very similar and, in fact, simpler, since one can use exact coupling instead of an approximate $\eps$-coupling.
In the boundary cases $p=0$ and $p=1$, the process $S$ becomes deterministic and is thus not non-arithmetic. In these cases, the limiting distribution is trivial, with $Q_{0}=F_{0}=\delta_{1}$ and $Q_{1}=F_{1}=\delta_{\vartheta}$, respectively. 

\vspace{.1cm}
Let $F_{p}^{*}$ denote the \emph{stationary renewal distribution} of $S$ under $\Pbf_{p}$, given by
\begin{equation*}\label{eq:stationary delay}
F_{p}^{*}(\dd x)\ =\ \mu(p)^{-1}\Pbf_{p}(\xi_{1}>x)\1_{(0,\infty)}(x)\,\dd x.
\end{equation*}
In the non-arithmetic case, $F_{p}^{*}$ is characterized by
\begin{equation*}
F_{p}^{*}*\Ubf_{p}=\mu(p)^{-1}\llam^{+}
\end{equation*}
where $\llam^{+}$ denotes Lebesgue measure on the positive halfline. It is also the limiting distribution of the overshoot $S_{\tau(\K)}-\K$ as $\K\to\infty$ and hence the stationary law of the continuous-time Markov process $(S_{\tau(\K)}-\K)_{\K\geq 0}$ under $\Pbf_{p}$.

\vspace{.2cm}
We now fix, as indicated above, an arbitrarily small $\nu\in (0,\frac{1}{2})$ and restrict attention to $p\in [\nu,1-\nu]$. Let $(\xi_{1}',\xi_{1}''),(\xi_{2}',\xi_{2}''),\ldots$ be iid~under every $\Pbf_{p}$ with common joint law defined by
\begin{gather*}
\Pbf_{p}(\xi_{1}'=\vartheta,\xi_{1}''=0)\,=\,\Pbf_{p}(\xi_{1}'=0,\xi_{1}''=\vartheta)=\frac{\nu}{4},\\
\Pbf_{p}(\xi_{1}'=1,\xi_{1}''=0)\,=\,\Pbf_{p}(\xi_{1}'=0,\xi_{1}''=1)\,=\,\frac{\nu}{4},\\
\Pbf_{p}(\xi_{1}'=\xi_{1}''=\vartheta)\,=\,\frac{p}{2}-\frac{\nu}{4},\quad \Pbf_{p}(\xi_{1}'=\xi_{1}''=1)\,=\,\frac{1-p}{2}-\frac{\nu}{4}
\shortintertext{and}
\Pbf_{p}(\xi_{1}'=\xi_{1}''=0)\,=\,\frac{1-\nu}{2}.
\end{gather*}
It follows that $\xi_{n}'$ and $\xi_{n}''$ have the same law under $\Pbf_{p}$, namely
$$ \frac{1}{2}\delta_{0}+\frac{p}{2}\delta_{\vartheta}+\frac{1-p}{2}\delta_{1}\ =\ \frac{1}{2}\big(\delta_{0}+F_{p}\big). $$
Moreover, the law of the difference $\xi_{n}'-\xi_{n}''$ is symmetric and independent of $p\in [\nu,1-\nu]$, namely
\begin{equation}\label{eq:symmetrization law}
\Pbf_{p}(\xi_{n}'-\xi_{n}''\in\cdot)\ =\ (1-\nu)\delta_{0}\,+\,\frac{\nu}{4}\big(\delta_{\vartheta}+\delta_{-\vartheta}+\delta_{1}+\delta_{-1}\big)
\end{equation}
for each $p\in [\nu,1-\nu]$. Note that this law is non-arithmetic under our assumption that $\vartheta\notin\Qb$.

\vspace{.2cm}
Now let $(\xi_{0}',\xi_{0}'')$ be independent of $(\xi_{n}',\xi_{n}'')_{n\geq 1}$ and distributed according to $\delta_{0}\otimes F_{p}^{*}$ under $\Pbf_{p}$. Define the bivariate random walk
\begin{gather*}
(S_{n}',S_{n}'')\,\coloneqq\,\sum_{k=0}^{n}(\xi_{k}',\xi_{k}''),\quad n\geq 0
\shortintertext{and its symmetrization}
W_{n}\,\coloneqq\,S_{n}''-S_{n}'\,=\,\sum_{k=0}^{n}(\xi_{k}''-\xi_{k}'),\quad n\geq 0.
\end{gather*}
By \eqref{eq:symmetrization law}, the law of the sequence $(W_{n}-W_{0})_{n \geq 0}$ under $\Pbf_{p}$ is the same for every $p \in [\nu, 1-\nu]$.

\vspace{.1cm}
Define the $\eps$-coupling time
$$ T_{\eps}\,=\,\inf\{n\geq 0:|W_{n}|\le\eps\}, $$
and more generally,
$$T_{\eps, x}\ =\ \inf \{ n \geq 0: \vert W_{n} -W_{0}+x \vert \leq \eps \}$$
for $\eps >0$ and $x >0$. Using $\Pbf_{\bullet}$ for probabilities that are independent of $p$, we then have
$$ \Pbf_{p}(T_\eps \in \cdot)\ =\ F_{p} (\{ \del\})\, \Pbf_{\bullet} (T_{\eps, \del}\in \cdot) + \big(1-F_{p}(\{\del\})\big)\, \Pbf_{\bullet}(T_{\eps, 1}\in \cdot). $$
This shows that the law of $T_\eps$ under $\Pbf_{p}$ depends on $p$ only through $F_{p}(\{ \del\})$, and is bounded by the larger of the two laws on the right-hand side, in the sense that
\begin{equation*}
\Pbf_{p}(T_\eps \in \cdot)\ \leq\ \Pbf_{\bullet}(T_{\eps, \del}\in \cdot) \vee \Pbf_{\bullet}(T_{\eps,1}\in \cdot).
\end{equation*}
We also observe that, if $\sigma_{0}'=0$ and
$$ \sigma_{n}'\,=\,\inf\{k>\sigma_{n-1}':\xi_{k}'>0\}\,=\,\inf\{k>\sigma_{n-1}':S_{k}'>S_{\sigma_{n-1}'}'\}$$
for $n\geq 1$ denote the jump epochs (strictly ascending ladder epochs) of $S'$, then the process $(S_{\sigma_{n}'}')_{n\geq 0}$ has the same law as the original process $S$ under every $\Pbf_{p}$. Furthermore, the increments of $(\sigma_{n}')_{n\geq 0}$ are iid~and geometrically distributed on $\Nb$ with parameter $\frac{1}{2}$. Now define
$$ \tau'(\K)\,=\,\inf\{n\geq 1:S_{n}'\geq\K\},\quad \K\geq 0. $$
Since level exceedance by $S'$ can only occur at a jump time, it follows that $\tau'(\K)=\sigma_{g(\K)}'$ for some suitable index function $g(\K)$.

\vspace{.2cm}
For $n\in\Nb_{0}$ and measurable $A\subset [0,\infty)$,  define the counting processes
\begin{align*}
N_{n}(A)\,\coloneqq\,\sum_{k=0}^{n}\1^{}_{A}(S_{k})\quad\text{and}\quad N(A)\,\coloneqq\,\sum_{k\geq 0}\1^{}_{A}(S_{k})
\end{align*}
and define $N_{n}'(A),N_{n}''(A),N'(A),N''(A)$ accordingly for $S'$ and $S''$, respectively. Then, by definition of the renewal measure, we have $\Ubf_{p}(A)=\Ebf_{p}[N(A)]$. The next lemma shows that augmenting the increment law by an atom at zero (i.e., replacing $F_{p}$ with $\frac{1}{2}(\delta_{0}+F_{p})$) changes the renewal measure only by a constant. We continue with some auxiliary lemmata used in the uniform coupling argument for the proof of Proposition \ref{uniform stopping summand}.

\begin{lemma}\label{lem:aux2}
Let $\Ubf_{p}'$ and $\Ubf_{p}''$ denote the renewal measures of $S'$ and $S''$ under $\Pbf_{p}$, respectively. Then
$$ \Ubf_{p}'\,=\,2\Ubf_{p}\quad\text{and}\quad\Ubf_{p}''\,=\,F_{p}^{*}*\Ubf_{p}'\,=\,2\mu(p)^{-1}\llam^{+} $$
for each $p\in (0,1)$.
\end{lemma}

\begin{proof}
Since $S''$ has the same increment law as $S'$ and delay distribution $F_{p}^{*}$, only the first identity needs to be verified. Let $\varphi_{p}$ denote the Laplace transform of $F_{p}$. Then the Laplace transform of $(\delta_{0}+F_{p})/2$ equals $(1+\varphi_{p})/2$. It follows that $\Ubf_{p}'=\sum_{n\geq 0}2^{-n}(\delta_{0}+F_{p})^{*n}$ has Laplace transform
$$ \frac{1}{1-(1+\varphi_{p})/2}\ =\ \frac{2}{1-\varphi_{p}} $$
which is also the Laplace transform of $2\Ubf_{p}$.
\end{proof}

\begin{lemma}\label{lem:aux3}
For all $\K\geq 0$, $n\in\Nb_{0}$ and $p\in [\nu,1-\nu]$,
\begin{gather}
N_{n}'\big([\K,\K+\vartheta)\big)\ \leq\ (\sigma_{g(\K)+1}'-\sigma_{g(\K)}')\1_{\{\tau'(\K)\leq n\}}\quad\Pbf_{p}\text{-a.s.}\label{eq:aux3.1}
\shortintertext{and}
\Ebf_{p}\big[N_{n}'\big([\K,\K+\vartheta)\big)\big]\ \le\ 2\,\Pbf_{p}(\tau'(\K)\leq n).\label{eq:aux3.2}
\end{gather}
\end{lemma}

\begin{proof}
By definition of $S'$, the walk can only visit the interval $[\K,\K+\vartheta)$ within $n$ steps if $\tau'(\K)=\sigma_{g(\K)}'\leq n$. In that case, $S'$ will exit the interval at the next positive jump, which is of size at least $\vartheta$. Thus, \eqref{eq:aux3.1} follows immediately. Since $\big(\sigma_{g(\K)+k}'-\sigma_{g(\K)}'\big)_{k\geq 0}$ and $\tau'(\K)$ are independent and $\Ebf_{p}\big[\sigma_{g(\K)+1}'-\sigma_{g(\K)}'\big]=2$, we obtain \eqref{eq:aux3.2} by taking expectations in \eqref{eq:aux3.1}.
\end{proof}

\begin{remark}\label{rem1:aux3}\rm
Since $S'$ differs from the original walk $S$ only by the inclusion of additional zero jumps, it is immediate that
\begin{equation}\label{eq:tau(a)>tau'(a)}
\Pbf_{p}\big(\tau'(\K) \leq n\big) \,\leq\, \Pbf_{p}\big(\tau(\K) \leq n\big)
\end{equation}
for all $\K \geq 0$, $n \in \Nb_{0}$, and $p \in [\nu, 1 - \nu]$. Moreover, we have $\tau(\K) \geq \K$ because the maximal jump size of $S$ is $1$ (see~\eqref{eq:stopping bounds}).
\end{remark}

\begin{remark}\label{rem2:aux3}\rm
Since
\begin{equation*}
\Ebf_{p}\big[N_{n}''\big([\K, \K + \vartheta)\big)\big] 
\ =\ \int_{(0,1]} \Ebf_{p}\big[N_{n}'\big([\K - x, \K + \vartheta - x)\big)\big] \, F_{p}^{*}(\dd x),
\end{equation*}
the previous lemma, combined with \eqref{eq:tau(a)>tau'(a)}, implies that
\begin{align*}
\Ebf_{p}\big[N_{n}''\big([\K, \K + \vartheta)\big)\big]\  
&\leq\ \int_{(0,1]} 2\, \Pbf_{p}\big(\tau'(\K - x) \leq n\big)\, F_{p}^{*}(\dd x)\ \leq\ 2\, \Pbf_{p}\big(\tau(\K - 1)\leq n\big)
\end{align*}
for all $\K \geq 0$, $n \in\Nb_{0}$, and $p \in [\nu,1-\nu]$.
\end{remark}

To formulate the next lemma, let $(\Fs_{n})_{n \geq 0}$ denote the canonical filtration of the bivariate random walk $(S_{n}', S_{n}'')_{n \geq 0}$. Note that the $\eps$-coupling time is a stopping time with respect to this filtration.

\begin{lemma}\label{lem:aux4}
For all $\eps > 0$, $\K \geq 0$, and $p \in [\nu, 1 - \nu]$,
\begin{equation}\label{eq:aux4.1}
\Ebf_{p}\big[N_{T_{\eps}}'\big([\K, \K + \vartheta)\big)\big]\ \leq\ 2\,\Pbf_{\bullet}(T_{\eps} \geq \K),
\end{equation}
where $\Pbf_{\bullet}$ indicates that this probability is independent of $p$. As a consequence,
\begin{equation}\label{eq:aux4.2}
\lim_{\K \to \infty} \sup_{p \in [\nu, 1 - \nu]} \Ebf_{p}\big[N_{T_{\eps}}'\big([\K, \K + \vartheta)\big)\big]\ =\ 0,
\end{equation}
and the same uniform convergence holds for $\Ebf_{p}\big[N_{T_{\eps}}''([\K, \K + \vartheta))\big]$.
\end{lemma}

\begin{proof}
From \eqref{eq:aux3.1}, we know that
$$ N_{T_\eps}'\big([\K, \K + \vartheta)\big)\ \leq\ (\sigma_{g(\K)+1}' - \sigma_{g(\K)}')\, \1_{\{\tau'(\K) \leq T_{\eps}\}}\quad \text{$\Pbf_{p}$-a.s. for all $p \in [\nu, 1 - \nu]$}. $$
Since $\{\tau'(\K) \leq T_{\eps}\} \in \Fs_{\tau'(\K)}$ and the increment $\sigma_{g(\K)+1}' - \sigma_{g(\K)}'$ is independent of $\Fs_{\tau'(\K)}$, it follows that
$$ \Ebf_{p}\big[N_{T_{\eps}}'\big([\K, \K + \vartheta)\big)\big] 
\,\leq\, 2\, \Pbf_{\bullet}\big(T_{\eps} \geq \tau'(\K)\big). $$
Using the fact that $\tau'(\K)\geq\K$ and that the law of $T_{\eps}$ under $\Pbf_{p}$ does not depend on $p$, we obtain \eqref{eq:aux4.1}. The remaining statements follow immediately.
\end{proof}

We now have all the ingredients to prove Proposition~\ref{uniform stopping summand}.

\begin{proof}[Proof of Proposition~\ref{uniform stopping summand}]
Fix any irrational $\vartheta$, and recall that
$$
Q_{p}^{\K}(\{\vartheta\})\ =\ \Pbf_{p}(\xi_{\tau(\K)}\ =\ \vartheta)\ =\ p\,\Ubf_{p}\big([\K - \vartheta, \K)\big)\ =\ 1 - \Pbf_{p}(\xi_{\tau(\K)}\ =\ 1).
$$
Further recalling \eqref{eq:form of Qp}, we obtain
$$
\big|Q_{p}^{\K}(\{\vartheta\}) - Q_{p}(\{\vartheta\})\big|\ =\ p \left| \Ubf_{p}\big([\K - \vartheta, \K)\big) - \frac{\vartheta}{\mu(p)} \right|.
$$
Hence, to show
$$
\sup_{p \in [0,1]} \left| Q_{p}^{\K}(\{\vartheta\}) - Q_{p}(\{\vartheta\}) \right| \xrightarrow{\K \to \infty} 0,
$$
it suffices to establish that $\Ubf_{p}([\K - \vartheta, \K)) \to \vartheta / \mu(p)$ uniformly in $p$. By Lemma~\ref{lem:aux 1}, we may restrict to $p \in [\nu, 1 - \nu]$ for arbitrary $\nu \in (0, \tfrac{1}{2})$. Thus, it remains to prove
\begin{equation}\label{eq:to be shown}
\lim_{\K \to \infty} \sup_{p \in [\nu, 1 - \nu]} \left| \Ubf_{p}\big([\K - \vartheta, \K)\big) - \frac{\vartheta}{\mu(p)} \right|\ =\ 0.
\end{equation}

To this end, fix any $\eps \in (0, \tfrac{\vartheta}{2})$. Observe that
$$ \Ubf_{p}\big([\K - \vartheta, \K)\big) - \frac{\vartheta}{\mu(p)}\ =\ \frac{1}{2} \Big( \Ubf_{p}'\big([\K - \vartheta, \K)\big) - \Ubf_{p}''\big([\K - \vartheta, \K)\big) \Big), $$
and recall that $\Ubf_{p}''\ =\ \frac{2}{\mu(p)} \llam^{+}$. We estimate
\begin{align*}
\Ubf_{p}'\big([\K - \vartheta, \K)\big)\  
&=\ \Ebf_{p} \big[N_{T_{\eps}}'\big([\K - \vartheta, \K)\big)\big] 
  +\ \Ebf_{p} \Bigg[ \sum_{n > T_{\eps}} \1_{[\K - \vartheta, \K)}(S_{n}') \Bigg] \\
&\leq\ \Ebf_{p} \big[N_{T_{\eps}}'\big([\K - \vartheta, \K)\big)\big] 
  +\ \Ebf_{p} \Bigg[ \sum_{n > T_{\eps}} \1_{[\K - \vartheta - \eps, \K + \eps)}(S_{n}'') \Bigg] \\
&=\ \Ebf_{p} \big[N_{T_{\eps}}'\big([\K - \vartheta, \K)\big)\big] 
   -\ \Ebf_{p} \big[N_{T_{\eps}}''\big([\K - \vartheta - \eps, \K + \eps)\big)\big] 
   +\ \Ubf_{p}''\big([\K - \vartheta - \eps, \K + \eps)\big) \\
&\leq\ o(1)\ +\ \frac{2(\vartheta + 2\eps)}{\mu(p)}\quad \text{as } \K \to \infty,
\end{align*}
where the $o(1)$ term is uniform in $p \in [\nu, 1 - \nu]$ by Lemma~\ref{lem:aux4}. Therefore,
\begin{equation}\label{eq:limsup estimate}
\limsup_{\K \to \infty} \sup_{p \in [\nu, 1 - \nu]} 
\left( \Ubf_{p}\big([\K - \vartheta, \K)\big) - \frac{\vartheta}{\mu(p)} \right)\ \leq\ \frac{2\eps}{\mu(1-\nu)}.
\end{equation}

A similar argument yields the lower bound
\begin{align*}
\Ubf_{p}'\big([\K - \vartheta, \K)\big) 
&\geq\ \Ebf_{p} \Bigg[ \sum_{n > T_{\eps}} \1_{[\K - \vartheta, \K)}(S_{n}') \Bigg] \\
&\geq\ \Ubf_{p}''\big([\K - \vartheta + \eps, \K - \eps)\big) 
  \  -\ \Ebf_{p} \big[N_{T_{\eps}}''\big([\K - \vartheta + \eps, \K - \eps)\big)\big] \\
&=\ \frac{2(\vartheta - 2\eps)}{\mu(p)}\ -\ o(1)\quad \text{as } \K \to \infty,
\end{align*}
again with uniform remainder $o(1)$ in $p \in [\nu, 1 - \nu]$. Thus,
\begin{equation}\label{eq:liminf estimate}
\liminf_{\K \to \infty} \inf_{p \in [\nu, 1 - \nu]} 
\left( \Ubf_{p}\big([\K - \vartheta, \K)\big) - \frac{\vartheta}{\mu(p)} \right)
\ \geq\ -\frac{2\eps}{\mu(1-\nu)}.
\end{equation}
Combining \eqref{eq:limsup estimate} and \eqref{eq:liminf estimate}, and noting that $\eps > 0$ was arbitrary, we obtain \eqref{eq:to be shown}. This completes the proof of Proposition~\ref{uniform stopping summand}.
\end{proof}

\subsection{Uniform $L^p$-convergence of $\K/\tau(\K)$}

Before giving the proof of Theorem~\ref{main}, we state the second announced result.

\begin{proposition} \label{prop:uniform key renewal}
For any $m \in \Nb$ and $\beta > 0$,
\begin{equation} \label{eq0:uniform key renewal}
\lim_{\K \to \infty} \sup_{p \in [0,1]} \left| \K^{m} \,\Ebf_{p}\left[ \frac{\xi_{\tau(\K)}^{\beta}}{\big(\mu(p)\tau(\K)\big)^{m}} \right] - \Ebf_{p}[\xi_{\infty}^{\beta}] \right|\ =\ 0,
\end{equation}
and additionally, for $\beta = 0$,
\begin{equation} \label{eq1:uniform key renewal}
\lim_{\K \to \infty} \sup_{p \in [0,1]} \Ebf_{p} \left[\left|\frac{R^{m}}{\big(\mu(p)\tau(\K)\big)^{m}}\,-\,1 \right| \right]\ =\ 0,
\end{equation}
which is equivalent to
\begin{equation} \label{eq2:uniform key renewal}
\lim_{\K\to\infty}\sup_{p \in [0,1]}\Ebf_{p} \left[\left| \frac{R}{\mu(p)\tau(\K)}\,-\,1\right|^m \right]\ =\ 0.
\end{equation}
\end{proposition}

\begin{proof}
For the proof of \eqref{eq0:uniform key renewal}, we note that
$$
\K^{m}\,\Ebf_{p}\left[\frac{\xi_{\tau(\K)}^{\beta}}{\big(\mu(p)\tau(\K)\big)^{m}}\right]\,-\,\Ebf_{p}[\xi_{\infty}^{\beta}]
\ =\ \Ebf_{p}\left[\xi_{\tau(\K)}^{\beta} \left( \frac{\K^{m}}{\big(\mu(p)\tau(\K)\big)^{m}}\,-\,1 \right) \right]\,+\,\Ebf_{p}[\xi_{\tau(\K)}^{\beta}]\,-\,\Ebf_{p}[\xi_{\infty}^{\beta}].
$$
Applying \eqref{eq1:uniform key renewal} and Proposition~\ref{uniform stopping summand}, we obtain
$$
\sup_{p \in [0,1]} \left| \Ebf_{p} \left[ \xi_{\tau(\K)}^{\beta} \left( \frac{\K^{m}}{\big(\mu(p)\tau(\K)\big)^{m}}\,-\,1 \right) \right] \right|
\ \leq\ \sup_{p \in [0,1]} \Ebf_{p} \left[ \left| \frac{\K^{m}}{\big(\mu(p)\tau(\K)\big)^{m}}\,-\,1 \right| \right]\ \xrightarrow{\K \to \infty}\ 0,
$$
and
$$
\sup_{p \in [0,1]} \left| \Ebf_{p}[\xi_{\tau(\K)}^{\beta}]\,-\,\Ebf_{p}[\xi_{\infty}^{\beta}]\right|\ \xrightarrow{\K \to \infty}\ 0.
$$
so \eqref{eq0:uniform key renewal} follows.

\vspace{0.2cm}

Since $\K / \tau(\K) \to \mu(p)$ $\Pbf_p$-a.s. for each $p$, the equivalence of \eqref{eq1:uniform key renewal} and \eqref{eq2:uniform key renewal} follows by a theorem of Riesz; see \cite[Theorem 15.4]{Bauer:01}. We now prove \eqref{eq2:uniform key renewal}. From \eqref{eq:stopping bounds}, we have for all $m \in \Nb$
$$
\Ebf_{p} \left[ \left| \frac{R}{\mu(p)\tau(\K)}\,-\,1 \right|^{m+1} \right]\ \leq\ C\,\Ebf_{p} \left[ \left| \frac{R}{\mu(p)\tau(\K)}\,-\,1 \right|^{m} \right].
$$
for some constant $C > 0$. Hence, it suffices to consider even $m$. Expanding the $m$-th power,
\begin{align*}
\Ebf_{p} \left[ \left( \frac{R}{\mu(p)\tau(\K)} - 1 \right)^m \right] 
\ =\ 1\,+\,\sum_{k=1}^m (-1)^{m-k} \binom{m}{k} \Ebf_{p} \left[ \frac{R^k}{\big(\mu(p)\tau(\K)\big)^k} \right].
\end{align*}
Using $\sum_{k=1}^m (-1)^{m-k} \binom{m}{k} = -1$, we see that it suffices to prove that for all $k \in \Nb$,
\begin{equation} \label{eq:Gerold 1a}
\lim_{\K \to \infty} \sup_{p \in [0,1]} \left| \Ebf_{p}\left[ \frac{R^k}{\big(\mu(p)\tau(\K)\big)^k} \right]\,-\,1 \right|\ =\ 0.
\end{equation}
To this end, we stipulate that all subsequent convergence statements (including big O symbols) are meant to hold uniformly in $p$. We expand $\tau(\K)^{-k}$ via Taylor's theorem around $\Ebf_{p}[\tau(\K)]$
\begin{align}\label{eq:Gerold 2}
\Ebf_{p} \left[ \frac{R^k}{\big(\mu(p)\tau(\K)\big)^k} \right] 
\ =\ \frac{R^k}{\big(\mu(p)\Ebf_{p}[\tau(\K)]\big)^k}\,+\, \frac{k(k+1)R^k}{2 \mu(p)^k} \cdot \Ebf_{p} \left[ \frac{\big(\tau(\K) - \Ebf_{p}[\tau(\K)]\big)^2}{\zeta^{k+2}} \right],
\end{align}
where $\zeta$ is between $\tau(\K)$ and $\Ebf_{p}[\tau(\K)]$. For suitable $0 < c_1 \le 1 \le c_2$, we have that
$$
\Pbf_{p}(c_1 R \le \zeta \le c_2 R)\ =\ 1 \quad \text{for all } R \ge 1, \; p \in [0,1].
$$
Thus,
\begin{align*}
\frac{\textbf{Var}_{p}[\tau(\K)]}{c_{2}^{k+2}R^{k+2}}\ \le\ \Ebf_{p}\bigg[\frac{\big(\tau(\K)-\Ebf_{p}[\tau(\K)]\big)^{2}}{\zeta^{k+2}}\bigg]\ \le\ \frac{\textbf{Var}_{p}[\tau(\K)]}{c_{1}^{k+2}R^{k+2}}.
\end{align*}
From Wald’s first identity,
\begin{equation}\label{eq.wald}
\mu(p)\Ebf_{p}[\tau(\K)]\ =\ \Ebf_{p}[S_{\tau(\K)}]\ \in\ [R,R+1]
\end{equation}
Using this, an application of Wald’s second identity yields
\begin{align*}
\mu(p)^{2}\,\textbf{Var}_{p}[\tau(\K)]\ &=\ \Ebf_{p}\Big[\Big((\mu(p)\tau(\K)-S_{\tau(\K)})+(S_{\tau(\K)}-\Ebf_{p}S_{\tau(\K)})\Big)^2\Big]\\
&=\ \Ebf_{p}\big[\big(S_{\tau(\K)}-\mu(p)\tau(\K)\big)^{2}\big]\,+\,\Ebf_{p}\big[\big(S_{\tau(\K)}-R+O(1)\big)^2\big]\\
&\hspace{1.5cm} -\,2\,\Ebf_{p}\big[\big(S_{\tau(\K)}-\mu(p)\tau(\K)\big)\big(S_{\tau(\K)}-R+O(1)\big)\big]\\
&=\ \textbf{Var}_{p}[\xi]\,\Ebf_{p}[\tau(\K)]\,+\,O(1)\,+\,O\Big(\sqrt{\textbf{Var}_{p}[\xi]\,\Ebf_{p}[\tau(\K)]}\,  \Big)\\
&=\ \textbf{Var}_{p}[\xi]\frac{R}{\mu(p)}\,+\,O(R^{1/2}),
\end{align*}
where the Cauchy-Schwarz inequality has been used for the last two equalities to deduce
\begin{align*}
\Big|\Ebf_{p}\Big[\big(S_{\tau(\K)}&-\mu(p)\tau(\K)\big)\big(S_{\tau(\K)}-R+O(1)\big)\Big]\Big|\\
&\le\ \sqrt{\Ebf_{p}\big[\big(S_{\tau(\K)}-\mu(p)\tau(\K)\big)^{2}\big]\,\Ebf_{p}\big[(S_{\tau(\K)}-R+O(1)\big)^{2}\big]}\\
&=\ \sqrt{\textbf{Var}_{p}[\xi]\,\Ebf_{p}[\tau(\K)]}\,O(1)\ =\ O(R^{1/2})\quad\text{as }R\to\infty.
\end{align*}
Returning to \eqref{eq:Gerold 2} and using \eqref{eq.wald}, we conclude
$$
\Ebf_{p} \left[ \frac{R^k}{\big(\mu(p)\tau(\K)\big)^k} \right]\ =\ \frac{R^k}{\big(\mu(p)\Ebf_{p}[\tau(\K)]\big)^k}\,+\,O(R^{-1}) \ =\ 1\,+\,O(R^{-1}) \quad \text{as } R \to \infty,
$$
uniformly in $p \in [0,1]$, proving \eqref{eq:Gerold 1a}, hence also \eqref{eq2:uniform key renewal}. This completes the proof of Proposition~\ref{prop:uniform key renewal}.
\end{proof}

\section{Proof of Theorem \ref{main}} \label{sec:proof main}

In view of the strategy outlined in Subsection~\ref{subsec:proof strategy via uniform}, we must verify conditions (\ref{eq.gen1}\,--\,\ref{eq.gen4}).

The condition
\begin{equation} \label{eq.rho}
\rho_{\K}(x)\ =\ x\,+\,\frac{\rho(x)}{\K}\,+\,o\left(\frac{1}{\K}\right) \quad \text{as } \K \to \infty
\end{equation}
from Theorem~\ref{main} implies that, for any $n \in \Nb$,
\begin{align*}
\Eb_{x}[(X_{1}^{\K}\,-\,x)^n]\  
&=\ \sum_{k=0}^{n} \binom{n}{k} \Eb_{x}\big[\big(X_{1}^{\K}\,-\,\rho_{\K}(x)\big)^k\big] \cdot \big(\rho_{\K}(x)\,-\,x\big)^{n-k} \\
&=\ \sum_{k=0}^{n} \binom{n}{k} \Eb_{x}\big[\big(X_{1}^{\K}\,-\,\rho_{\K}(x)\big)^k\big] \cdot \left( \frac{\rho(x)}{R} \right)^{n-k}\,+\,o(R^{-1}) \quad \text{as } \K \to \infty,
\end{align*}
where, throughout this section, all convergence statements involving $\Eb_{x}$ are understood to hold uniformly in $x \in [0,1]$, and those involving $\Ebf_{p}$ uniformly in $p \in [0,1]$.

As a consequence of the uniform convergence $R(\rho_{R}(x) - x) \to \rho(x)$, we deduce the following expansion, valid as $\K \to \infty$
\begin{equation} \label{moment formula}
\Eb_{x}[(X_{1}^{R}\,-\,x)^n]\ =\ n\,\Eb_{x}\big[\big(X_{1}^{R}\,-\,\rho_{R}(x)\big)^{n-1}\big] \cdot \frac{\rho(x)}{R} 
\,+\,\Eb_{x}[(X_{1}^{R}\,-\,\rho_{R}(x))^n]\,+\,o(R^{-1}).
\end{equation}

Using \eqref{eq:moments->renewal theory} together with the identity $\sum_{k=0}^{n} \binom{n}{k} (-1)^k = 0$, we find
\begin{align}
\Eb_{x}&\big[\big(X_{1}^{\K} - \rho_{\K}(x)\big)^{n}\big]\ =\ \frac{(-1)^n}{(1 - \del)^n} \, \Eb_{x}\Bigg[\bigg( \frac{S_{\tau(\K)}}{\tau(\K)} - \mu\big(\rho_{\K}(x)\big) \bigg)^n\Bigg] \label{eq2:moment formula}\\
&=\ \frac{(-1)^n}{(1 - \del)^n} \sum_{k=0}^{n} \binom{n}{k} \Eb_{x} \bigg[ \bigg( \frac{S_{\tau(\K)}}{\tau(\K)} \bigg)^k \bigg] \big(-\mu(\rho_{\K}(x))\big)^{n-k} \nonumber\\
&=\ \frac{(-1)^n}{(1 - \del)^n} \sum_{k=1}^{n} \binom{n}{k} \big(-\mu(\rho_{\K}(x))\big)^{n-k} 
\bigg( \Eb_{x} \bigg[ \bigg( \frac{S_{\tau(\K)}}{\tau(\K)} \bigg)^k \bigg] - \mu\big(\rho_{\K}(x)\big)^k \bigg).\nonumber
\end{align}

Combining \eqref{eq2:moment formula} with \eqref{moment formula} will pave the way for the proof of Theorem~\ref{main}, which is presented at the end of this section.

\begin{proposition}\label{prop:Stau over tau moments}
For any $m \in \Nb$, we have
\begin{equation} \label{eq:moment m}
\K \bigg( \Ebf_{p} \bigg[ \bigg( \frac{S_{\tau(\K)}}{\tau(\K)} \bigg)^m \bigg]\,-\,\mu(p)^m \bigg)\ =\ \frac{m(m+1)}{2} \mu(p)^{m-1} \, \mathbf{Var}_{p}[\xi]\,+\,O\big(R^{-1}\big),
\end{equation}
where the $O(R^{-1})$ term is uniform in $p \in [0,1]$.
\end{proposition}

For the proof of this result, we require the following auxiliary lemma, which provides a somewhat tedious but useful expansion for the integral moments of the ratio $\tau(\K)^{-1} S_{\tau(\K)}$.

\begin{lemma}\label{lemma_cases new}
In the given notation, for all $p \in [0,1]$, $m \in \Nb$, and $\beta \ge 0$, we have
\begin{align}
\Ebf_{p}\bigg[\bigg(\frac{S_{\tau(\K)}}{\tau(\K)}\bigg)^{m}\bigg]
\ =\ \sum_{k=0}^{m} \sum_{\substack{\alpha_{1},\ldots,\alpha_{k} \ge 1,\ \beta \ge 0 \\ \alpha_{1}\,+\,\cdots\,+\,\alpha_{k}\,+\,\beta = m}} 
\frac{m!}{\alpha_{1}!\cdots\alpha_{k}! \, \beta! \, k!} \, J_{m,k}^{(p)}(\alpha_{1},\ldots,\alpha_{k} \mid \beta), \label{eq:cases_{n}ew}
\end{align}
where, with $s_k \coloneqq x_1+\cdots+x_k$, the term $J_{m,k}^{(p)}$ is defined as
\begin{align*}
&J_{m,k}^{(p)}(\alpha_{1},\ldots,\alpha_{k} \mid \beta)\\ 
&\coloneqq\ \int \cdots \int \Ebf_{p} \Bigg[\xi_{\tau(\K - s_k)}^{\beta} \cdot \frac{ \prod_{j=0}^{k-1} \big(\tau(\K - s_k)\,+\,j\big)}{ \big(\tau(\K - s_k)\,+\,k\big)^{m} }\Bigg]\Bigg(\prod_{j=1}^{k} x_j^{\alpha_j} \Bigg)\ F_{p}(\mathrm{d}x_k) \cdots F_{p}(\mathrm{d}x_1),
\end{align*}
with the convention that when $k = 0$, the term reduces to
$$ J_{m,0}^{(p)}(m)\ =\ \Ebf_{p}\bigg[\bigg( \frac{\xi_{\tau(\K)}}{\tau(\K)} \bigg)^m \bigg]. $$
\end{lemma}

\begin{proof}
Let $n\in\Nb$. By the multinomial theorem we obtain
\begin{align*}
\Ebf_{p}&\bigg[\1_{\{\tau(\K)=n\}}\bigg(\frac{S_{\tau(\K)}}{\tau(\K)}\bigg)^{m}\bigg]\\
&=\ \sum_{k=0}^{m}\sum_{\alpha_{1},\ldots,\alpha_{k}\ge 1|\beta\ge 0\atop \alpha_{1}+\ldots+\alpha_{k}+\beta=m}\frac{m!}{\alpha_{1}!\cdots\alpha_{k}!\beta!}\sum_{1\le i_{1}<\ldots<i_{k}<n}\Ebf_{p}\Bigg[\1_{\{S_{n-1}<\K\le S_{n}\}}\Bigg(\prod_{j=1}^{k}\xi_{i_{j}}^{\alpha_{j}}\Bigg)\cdot\frac{\xi_{n}^{\beta}}{n^{m}}\Bigg]\\
&=\ \sum_{k=0}^{m}\sum_{\alpha_{1},\ldots,\alpha_{k}\ge 1|\beta\ge 0\atop \alpha_{1}+\ldots+\alpha_{k}+\beta=m}\frac{m!}{\alpha_{1}!\cdots\alpha_{k}!\beta!}\,\binom{n-1}{k}J_{m,k,n}^{(p)}(\alpha_{1},\ldots,\alpha_{k}|\beta),
\end{align*}
where
\begin{gather*}
J_{m,k,n}^{(p)}(\alpha_{1},\ldots,\alpha_{k}|\beta)\ \coloneqq\ \int\!\!\cdots\!\!\int\Ebf_{p}\Bigg[\1_{\{S_{n-1-k}<\K-s_{k}\le S_{n-k}\}}\frac{\xi_{n-k}^{\beta}}{n^{m}}\Bigg]\Bigg(\prod_{j=1}^{k} x_{j}^{\alpha_{j}}\Bigg)\,F_{p}(\mathrm{d}x_{k})\ldots F_{p}(\mathrm{d}x_{1})\\
=\ \int\!\!\cdots\!\!\int\Ebf_{p}\Bigg[\1_{\{\tau(\K-s_{k})=n-k\}}\frac{\xi_{\tau(\K-s_{k})}^{\beta}}{(\tau(\K-s_{k})+k)^{m}}\Bigg]\Bigg(\prod_{j=1}^{k} x_{j}^{\alpha_{j}}\Bigg)\,F_{p}(\mathrm{d}x_{k})\ldots F_{p}(\mathrm{d}x_{1}).
\end{gather*}
Summing over all $n \in \Nb$ yields \eqref{eq:cases_{n}ew}.
\end{proof}

The next auxiliary lemma provides the final ingredient in the proof of Proposition~\ref{prop:Stau over tau moments}, namely suitable asymptotic expansions for the functions $J_{m,k}^{(p)}$.

\begin{lemma}\label{Jasymp}
For any $m,k \in \Nb$ with $k \le m-2$, and for $\alpha_1, \ldots, \alpha_k \in \Nb$, $\beta \ge 0$ satisfying $\alpha_1 + \cdots + \alpha_k + \beta = m$, we have
\begin{equation}\label{Jsmall}
J_{m,k}^{(p)}(\alpha_{1},\ldots,\alpha_{k}|\beta)\ =\ O(R^{-2}).
\end{equation}
Moreover, for each $m \in \Nb$, the following asymptotics hold
\begin{gather}
J_{m,m}^{(p)}(1,1,\ldots,1|0) \ =\ \mu(p)^m \left( 1 - \frac{m(m+1)}{2R} \mu(p) \right)\,+\,O(R^{-2}), \label{J10} \\
J_{m,m-1}^{(p)}(1,1,\ldots,1|1)\ =\ \frac{\mu(p)^{m-1}\,\Ebf_{p}[\xi^2]}{R}\,+\, O(R^{-2})\ =\ J_{m,m-1}^{(p)}(2,1,\ldots,1|0) \label{J11}. 
\end{gather}
\end{lemma}

\begin{proof}
Note first that
\begin{align*}
J_{m,k}^{(p)}(\alpha_{1},\ldots,\alpha_{k}|\beta)\ \le\ \Ebf_{p}\Bigg[ 
\prod_{j=0}^{k-1} \left(1 - \frac{k-j}{\tau(\K - k) + k} \right)\cdot \frac{1}{\big(\tau(\K - k) + k\big)^{m - k}} \Bigg]\ \le\ (R - m)^{-(m - k)},
\end{align*}
which is of order $O(R^{-2})$ for $k \le m - 2$, uniformly in $p$, hence proving \eqref{Jsmall}.
To establish \eqref{J10} and \eqref{J11}, set $\tau_k \coloneqq \tau(R-s_k)$. By definition of the functions $J_{m,k}^{(p)}$, we obtain
\begin{equation*}
    J_{m,m}^{(p)}(1,1,\ldots, 1\vert 0) = \int \cdots \int \mathbf{E}_p\bigg[ \frac{\prod_{j=0}^{m-1} (\tau_m+j)}{(\tau_m+m)^m} \bigg] \Big(\prod_{j=1}^m x_j\Big) F_p(\dd x_1) \ldots F_p(\dd x_m).
\end{equation*}
By combining the asymptotic expansion for large $x$
\begin{equation*}
\frac{\prod_{j=0}^{m-1} (x+j)}{(x+m)^m}=1-\frac{m(m+1)}{2 (x+m)}+ O\big((x+m)^{-2}\big),    
\end{equation*}
with Proposition \ref{prop:uniform key renewal}, we get
\begin{equation*}
\mathbf{E}_p\bigg[ \frac{\prod_{j=0}^{m-1} (\tau_m+j)}{(\tau_m+m)^m} \bigg] = 1-\frac{m(m+1)\mu(p)}{2R}+O(R^{-2}), 
\end{equation*}
which proves \eqref{J10}. Similarly, 
\begin{equation*}
    \mathbf{E}_p\bigg[ \xi_{\tau_{m-1}}\frac{\prod_{j=0}^{m-2}(\tau_{m-1}+j)}{(\tau_{m-1}+m-1)^m}\bigg] =  \mathbf{E}_p \bigg[ \frac{\xi_{\tau_{m-1}}}{\tau_{m-1}+m-1} \bigg] + O\big( ( \tau_{m-1}+m-1)^{-2}\big),
\end{equation*}
and, by applying Proposition \ref{uniform stopping summand}, we conclude that
\begin{align*}
    J_{m,m-1}^{(p)}(1,1,\ldots,1 \vert 1) =& \int \cdots \int \mathbf{E}_p\bigg[ \xi_{\tau_{m-1}} \frac{\prod_{j=0}^{m-2}(\tau_{m-1}+j)}{(\tau_{m-1}+m-1)^m} \bigg] \Big( \prod_{j=1}^{m-1} x_j \Big) F_p(\dd x_1) \cdots F_p(\dd x_{m-1})
    \\
    =&\frac{\mu(p)^{m-1} \mathbf{E}_p[\xi_\infty]}{R}+O(R^{-2})=\frac{\mu(p)^{m-1}\,\Ebf_{p}[\xi^2]}{R}\,+\, O(R^{-2}).
\end{align*}
The asymptotic expansion for $J_{m,m-1}^{(p)}(2,1,\ldots,1 \vert 0)$ follows analogously, which completes the proof.
\end{proof}

\begin{proof}[Proof of Proposition~\ref{prop:Stau over tau moments}]
Recall that all asymptotic expansions stated below are understood to hold uniformly in $p$. Note also that for any $m \in \Nb$, the symmetry of the integrals implies
$$ J_{m,m-1}^{(p)}(2,1,\ldots,1|0)\ =\ J_{m,m-1}^{(p)}(1,2,1,\ldots,1| 0)\ = \cdots 
=\ J_{m,m-1}^{(p)}(1,\ldots,1,2| 0). $$
Using Lemmata~\ref{lemma_cases new} and \ref{Jasymp}, we expand
\begin{align*}
\Ebf_{p}\Bigg[\bigg(\frac{S_{\tau(\K)}}{\tau(\K)}\bigg)^{m}\Bigg] 
&=\ m \cdot J_{m,m-1}^{(p)}(1,\ldots,1|1) 
\,+\,\frac{m(m-1)}{2} \cdot J_{m,m-1}^{(p)}(2,1,\ldots,1 \mid 0) \\
&\quad +\,J_{m,m}^{(p)}(1,\ldots,1|0) \,+\,O(R^{-2}) \\
&=\ \frac{m \mu(p)^{m-1} \, \Ebf_{p}[\xi^2]}{R} 
\,+\,\frac{m(m-1) \mu(p)^{m-1} \, \Ebf_{p}[\xi^2]}{2R} \\
&\quad +\, \mu(p)^m \left(1 - \frac{m(m+1)}{2R} \mu(p) \right) 
\,+\,O(R^{-2}) \\
&=\ \mu(p)^m 
\,+\,\frac{m(m+1)\mu(p)^{m-1}}{2R} 
\left( \Ebf_{p}[\xi^2]\,-\,\mu(p)^2 \right) 
\,+\,O(R^{-2}),
\end{align*}
which completes the proof.
\end{proof}

We are now ready to prove the main result.

\begin{proof}[Proof of Theorem~\ref{main}]
As outlined in Section~\ref{subsec:proof strategy via uniform}, the first step is to verify conditions (\ref{eq.gen1}\,--\,\ref{eq.gen4}). All convergence statements below are understood to hold uniformly in $x \in [0,1]$.

We begin with the first moment. Setting $n = 1$ in \eqref{moment formula}, and using \eqref{eq2:moment formula} together with Proposition~\ref{prop:Stau over tau moments}, we find
\begin{align} \label{first_moment}
R\,\Eb_{x}[X_{1}^{R} - x]\ &=\ - \frac{\textbf{Var}_{\rho_{R}(x)}[\xi]}{1 - \del}\,+\,\rho(x)\,+\,o(1) 
\ =\  - \frac{\textbf{Var}_{x}[\xi]}{1\,-\, \del}\,+\,\rho(x)\,+\,o(1),
\end{align}
where the second identity follows from the uniform convergence in \eqref{eq.rho}. This establishes \eqref{eq.gen1}.

\vspace{.1cm}
Now consider higher moments $(n\geq 2)$. Combining \eqref{moment formula}, \eqref{eq2:moment formula}, \eqref{first_moment}, and Proposition~\ref{prop:Stau over tau moments}, we obtain
\begin{align} \label{eq.moments}
R\,&\Eb_{x}[(X_{1}^{R}-x)^n]  
\ =\ R \, \Eb_{x}\big[\big(X_{1}^{R}-\rho_{R}(x)\big)^n\big]\,+\,o(1) \nonumber \\
&=\ \frac{(-1)^n}{(1-\del)^n} \sum_{k=1}^{n} 
\binom{n}{k} \big(-\mu(\rho_{R}(x))\big)^{n-k} \cdot \frac{k(k+1)}{2} 
\big(\mu(\rho_{R}(x))\big)^{k-1} \, \textbf{Var}_{\rho_{R}(x)}[\xi]\,+\,o(1).
\end{align}

For $n = 2$, this yields
\begin{align*}
R \, \Eb_{x}[(X_{1}^{R} - x)^2]\ &=\ \frac{1}{(1 - \del)^2} 
\left( -2 \mu(\rho_{R}(x)) \textbf{Var}_{\rho_{R}(x)}[\xi] 
\,+\,3 \mu(\rho_{R}(x)) \textbf{Var}_{\rho_{R}(x)}[\xi] \right) \,+\,o(1) \\
\ &=\ \frac{\mu(\rho_{R}(x)) \textbf{Var}_{\rho_{R}(x)}[\xi]}{(1 - \del)^2} \,+\,o(1) \\
&=\ \frac{\mu(x) \textbf{Var}_{x}[\xi]}{(1 - \del)^2} \,+\,o(1),
\end{align*}
using again \eqref{eq.rho}. This establishes \eqref{eq.gen2}.

\vspace{.1cm}
Next, for $n = 3$, we find
\begin{align*}
R \, \Eb_{x}[(X_{1}^{R} - x)^3]\ &=\ \frac{1}{(1 - \del)^3} 
\Big( 3 \mu(\rho_{R}(x))^2 \textbf{Var}_{\rho_{R}(x)}[\xi] 
\,-\,9 \mu(\rho_{R}(x))^2 \textbf{Var}_{\rho_{R}(x)}[\xi] \\
&\quad +\, 6 \mu(\rho_{R}(x))^2 \textbf{Var}_{\rho_{R}(x)}[\xi] \Big)\,+\,o(1)\ =\ o(1),
\end{align*}
confirming \eqref{eq.gen3}.

\vspace{.1cm}
Finally, for $n = 4$, we compute
\begin{align*}
R\,\Eb_{x}[(X_{1}^{R} - x)^4]\ &=\ \frac{1}{(1 - \del)^4} \Big( 
-4 \mu(\rho_{R}(x))^3 \textbf{Var}_{\rho_{R}(x)}[\xi] 
\,+\,18 \mu(\rho_{R}(x))^3 \textbf{Var}_{\rho_{R}(x)}[\xi] \\
&\quad - 24 \mu(\rho_{R}(x))^3 \textbf{Var}_{\rho_{R}(x)}[\xi] 
\,+\,10 \mu(\rho_{R}(x))^3 \textbf{Var}_{\rho_{R}(x)}[\xi] 
\Big) \,+\,o(1)\ =\ o(1),
\end{align*}
which proves \eqref{eq.gen4}.

\vspace{.1cm}
Using the Taylor expansion argument from Section~\ref{subsec:proof strategy via uniform} and combining with (\ref{eq.gen1}\,--\,\ref{eq.gen4}), we conclude that for any $f \in C^4([0,1])$,
$$ \lim_{R \to \infty} \sup_{x \in [0,1]} 
\lvert \mathcal{A}^R f(x) - \mathcal{A} f(x) \rvert\ =\ 0. $$

Since the diffusion coefficient $x \mapsto a(x) = x(1 - x)(1 - (1 - \vartheta)x)$ is non-negative, twice continuously differentiable, and vanishes at $x\in\{0,1\}$, and since the drift term $x \mapsto d(x) = -(1 - \vartheta)x(1 - x) + \rho(x)$ is Lipschitz continuous with $d(0) = \rho(0) = \beta_{0} \geq 0$ and $d(1) = \rho(1) = -\beta_1 \leq 0$, we may invoke \cite[Chap.~8, Thm.~2.1]{EK86} to conclude that $X$ is Feller and that $C^\infty([0,1])$ is a core for its generator $\mathcal{A}$. The result then follows from \cite[Theorem~1.6.1]{EK86}.
\end{proof}

\section{The model variant revisited} \label{sec:variant}

We now briefly return to the model variant described in Subsection~\ref{subsec:model variant}, where the stopping rule for each generation is to reject the first individual that would cause an overshoot of the available resources. We have already noted that if $\rho(x) \equiv 0$, the limiting diffusion model given by Theorem~\ref{thm:model variant} is neutral. The intuitive reason for this is the absence of the effect of the last individual. Specifically, the only reason small individuals experience a disadvantage in the original two-size Wright--Fisher model is the size-biased law of the stopping summand $\xi_{\tau(\K)}$, which does not apply in the variant, as the individual associated with $\xi_{\tau(\K)}$ is rejected.

The proof of Theorem~\ref{thm:model variant} is very similar to the proof of Theorem~\ref{main}, but instead of $\tau(\K)$, it requires considering the modification
$$
\overline{\tau}(\K) \coloneqq \inf \{ n \in \mathbb{N} : S_n > \K \}.
$$
As the counterpart to~\eqref{l1}, we then have
\begin{equation} \label{l2}
\Pb_{x}\big((\bX_{1}^{\K},\bM_{1}^{\K})\in\cdot\big)\ =\ \Pb_{x}\Bigg(\bigg(\frac{1}{1-\del}\Big(1-\frac{S_{\btau(\K)-1}}{\btau(\K)-1}\Big),\btau(\K)-1\bigg)\in\cdot\Bigg).
\end{equation}
With the help of this relation, the expression $\K \, \mathbb{E}_x [ (\bX_1^{\K} - x)^k]$ as $\K \to \infty$ can, for $k \in \{1, 2, 3, 4\}$, be analyzed in the same way as in the previous section, without the need for new arguments. For $k \in \{2, 3, 4\}$, the same results as in the original model are obtained, and for $k = 1$, we even have an explicit result, as the following lemma shows.

\begin{lemma}\label{lem:zero drift model variant}
For any fixed $\del \in (0,1)$ and $\rho(x) = 0$, 
\begin{equation*}
\mathbb{E}_{x}\big[\bX_{1}^{R} - x\big] = 0.
\end{equation*}
\end{lemma}

\begin{proof}
Since, by \eqref{l2},
\begin{equation*}
\mathbb{E}_{x}\big[ \bX_{1}^{R} - x \big] 
\ =\ \frac{1}{1 - \del} \left(\mu\big(\rho_{R}(x)\big)\,-\,\mathbb{E}_{x}\left[ \frac{S_{\overline{\tau}(R) - 1}}{\overline{\tau}(R) - 1}\right] \right),
\end{equation*}
it suffices to show that, for any $p \in [0,1]$,
\begin{equation}\label{eq:Stau-1}
\mathbb{E}_{p} \left[ \frac{S_{\overline{\tau}(R) - 1}}{\overline{\tau}(R) - 1} \right]\ =\ \mu(p).
\end{equation}
To this end, let $\xi$ be a generic copy of the $\xi_{i}$, independent of all other random variables under each $\mathbb{P}_{p}$. Then,
\begin{align*}
\mathbb{E}_{p} \left[ \frac{S_{\overline{\tau}(R) - 1}}{\overline{\tau}(R) - 1} \right]\  
&=\ \sum_{n\ge 1} \frac{1}{n} \sum_{i=1}^{n} \mathbb{E}_{p} \big[ \xi_{i}\,\mathbf{1}_{\{\overline{\tau}(R) - 1 = n\}}\big] \\
&=\ \sum_{n\ge 1} \frac{1}{n} \sum_{i=1}^{n} \mathbb{E}_{p} \big[\xi\,\mathbf{1}_{\{\bT(R - \xi) = n\}}\big] \ =\ \sum_{n \ge 1}\mathbb{E}_{p} \big[\xi\,\mathbf{1}_{\{ \bT(R - \xi) = n \}}\big]\ =\  \mu(p),
\end{align*}
as required.
\end{proof}

\medskip

Let us note in passing that Lemma~\ref{lem:zero drift model variant} can also be derived by observing that the sequence $(n^{-1} S_{n})_{n \ge 1}$ forms a reverse martingale, and that $\overline{\tau}(R)-1=\sup \{n\ge 0:S_{n}\le R\}$ is an associated reverse stopping time. This implies, see e.g.\ \cite[p.\,350]{Alsmeyer:1988} for more details, that 
$$ \mathbb{E}_{p} \left[ \frac{S_{\overline{\tau}(R) - 1}}{\overline{\tau}(R) - 1} \right]\ =\  \mathbb{E}_{p}[S_1]\ =\ \mu(p), $$
and thus we obtain \eqref{eq:Stau-1} once again.

\section{Brief Discussion of the Long-Term Behavior} \label{sec:asympt}

We conclude with a brief discussion of the long-term behavior of the solution to SDE~\eqref{eq.SDE} and its interpretation in the context of the underlying two-size Wright--Fisher model. This analysis does not require new theoretical developments, but instead relies on standard methods, as described in \cite{Durrett:08} and \cite{Etheridge:11}.

A key object in characterizing the long-term behavior is the \emph{scale function} $S(x)$, defined by
\begin{equation} \label{eq.scale}
S(x) \coloneqq \int_{x_{0}}^{x} \exp\left( -\int_{\eta}^{y} \frac{d(z)}{\sigma^{2}(z)} \,\dd z \right) \dd y,
\end{equation}
where $d$ and $\sigma$ denote the drift and diffusion coefficients, respectively, of the SDE under study, and $x_{0}, \eta$ are arbitrary points in the interval $(0,1)$. In what follows, we use the scale function to describe certain aspects of the long-term behavior of the SDE~\eqref{eq.SDE}.

We begin with the extinction probability of the small individuals as a function of their initial proportion, which is meaningful only in the absence of mutation, i.e., when $\rho(x) = s(x)\,x(1 - x)$ for some Lipschitz function $s : [0,1] \to \mathbb{R}$. 
With this in mind, define $T_a$, for $a \in \{0,1\}$, as the first hitting time of $a$ by the process $X$. The case $a = 0$ corresponds to extinction, and $a = 1$ to fixation. By standard results for one-dimensional SDEs (see \cite[Lemma 3.14]{Etheridge:11}), we have
\begin{equation*}
\mathbb{P}_{x}(T_{0} < T_1) = \frac{S(1) - S(x)}{S(1) - S(0)},
\end{equation*}
where $S(x)$ denotes the scale function. Substituting the drift and diffusion coefficients from the SDE~\eqref{eq.SDE} into the definition~\eqref{eq.scale} of the scale function yields
\begin{equation*}
S(x)\ =\ \int_{x_{0}}^{x} \exp\left( 
 2 (1 - \del) \int_{\eta}^{y} \frac{1}{1 - (1 - \del) z} \, \mathrm{d}z 
- 2 \int_{\eta}^{y} \frac{s(z)}{1 - (1 - \del) z} \, \mathrm{d}z
\right) \, \mathrm{d}y.
\end{equation*}
The first integral with respect to $z$ can always be computed explicitly; the second depends on the specific form of the selection function $s$. In the case of genic selection, i.e., $\rho(x) = s\,x(1 - x)$ with constant $s$, the extinction probability is given by
\begin{equation} \label{eq.extinction}
\mathbb{P}_{x}(T_{0} < T_1)\ =\  
\begin{cases}
\displaystyle\frac{\del^{-1 + 2s(1 - \del)^{-1}} - \big(1 - (1 - \del)x\big)^{-1 + 2s(1 - \del)^{-1}}}{\del^{-1 + 2s(1 - \del)^{-1}} - 1} & \text{if } 1 - \del \neq 2s, \\[10pt]
\displaystyle\frac{\ln(\del) - \ln\big(1 - (1 - \del)x\big)}{\ln(\del)} & \text{if } 1 - \del = 2s.
\end{cases}
\end{equation}

Figure~\ref{fig.fixation2} illustrates the extinction probability \eqref{eq.extinction} as a function of $x$, the initial proportion of small individuals, for various values of $s$ and $\del$. In the absence of selection ($s = 0$), the extinction probability exceeds $1 - x$, and this disadvantage increases as the size parameter $\del$ decreases. When $s = 1 - \del$, the model becomes neutral, and the extinction probability equals $1 - x$ (see Figure~\ref{subfig}, where $s = \del = 0.5$).

Interestingly, for $s \in \{1.5, 2\}$ and sufficiently large $x$, the extinction probability no longer decreases with increasing $\del$. To explain this, assume first that $\del \in [0,1]$ with $s < 1$. In this regime, the drift term 
$$ d(x)\ =\ \big(-(1 - \del) + s\big) \, x(1 - x) $$
decreases as $\del$ decreases. At the same time, the diffusion coefficient decreases as well, further limiting the process's deviation from its drift. The combined effect -- a stronger push toward $0$ for $\del < 1 - s$ and a weaker push toward $1$ for $\del > 1 - s$, along with reduced stochastic noise -- leads to an increased likelihood of extinction.

In contrast, when $s > 1$ and $\del \in [0,1]$, the drift is always positive, pushing the process toward $1$. Decreasing $\del$ weakens this drift, which might suggest, as before, that extinction becomes more likely. However, in this case the diffusion coefficient contains the additional factor $1 - (1 - \del)x$, which vanishes as $(1 - \del)x \to 1$. Consequently, if the process starts close to $1$, the reduced noise -- despite the weaker drift -- makes it harder to escape the vicinity of $1$. This results in a lower probability of reaching $0$.

\begin{figure}[h!]\centering 
\begin{subfigure}[b]{0.49\textwidth}
	\includegraphics[width=0.95\textwidth]{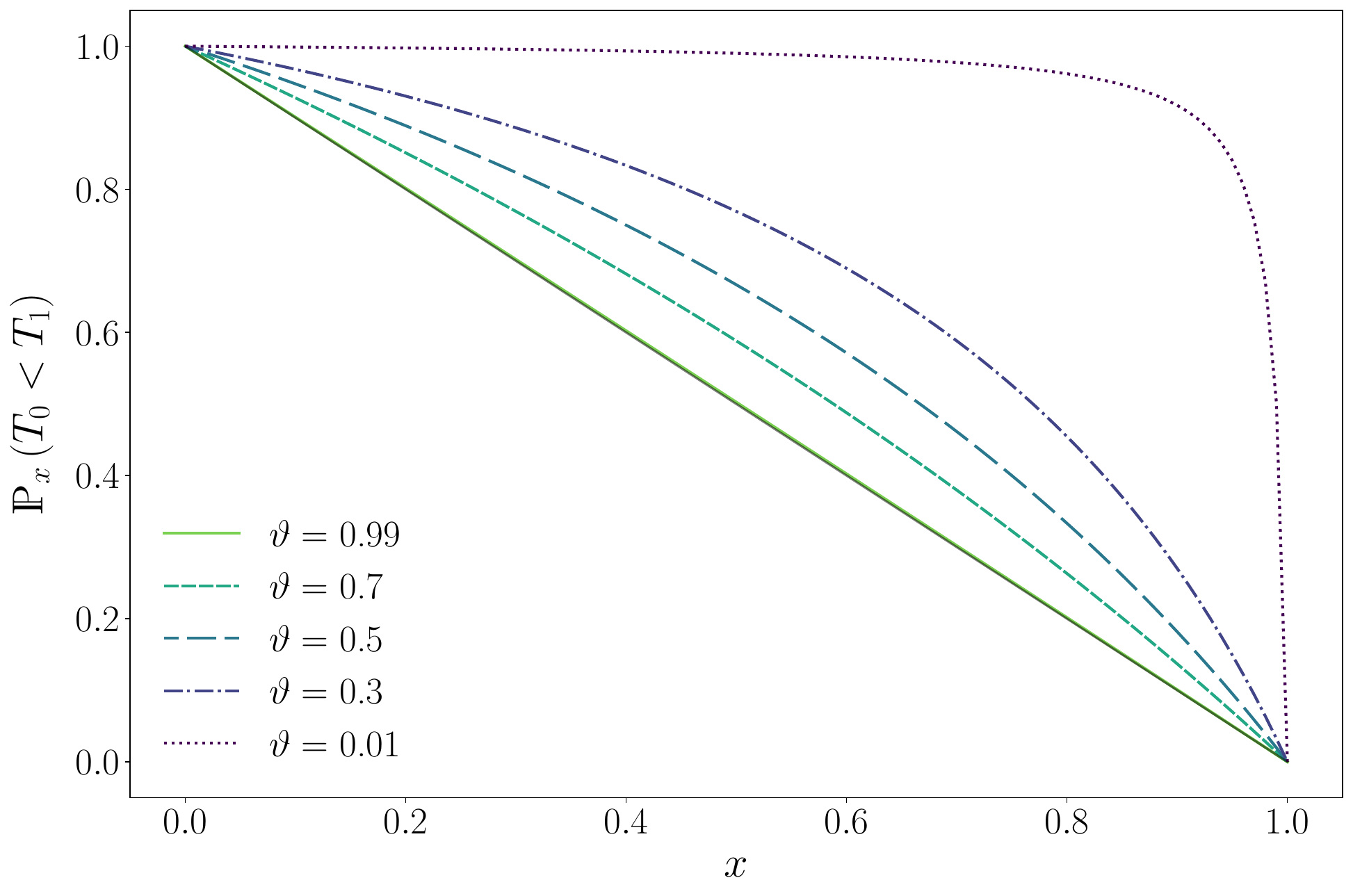}
	\caption{$s = 0$}
	\label{subfig0} 
\end{subfigure}
\begin{subfigure}[b]{0.49\textwidth}
	\includegraphics[width=0.95\textwidth]{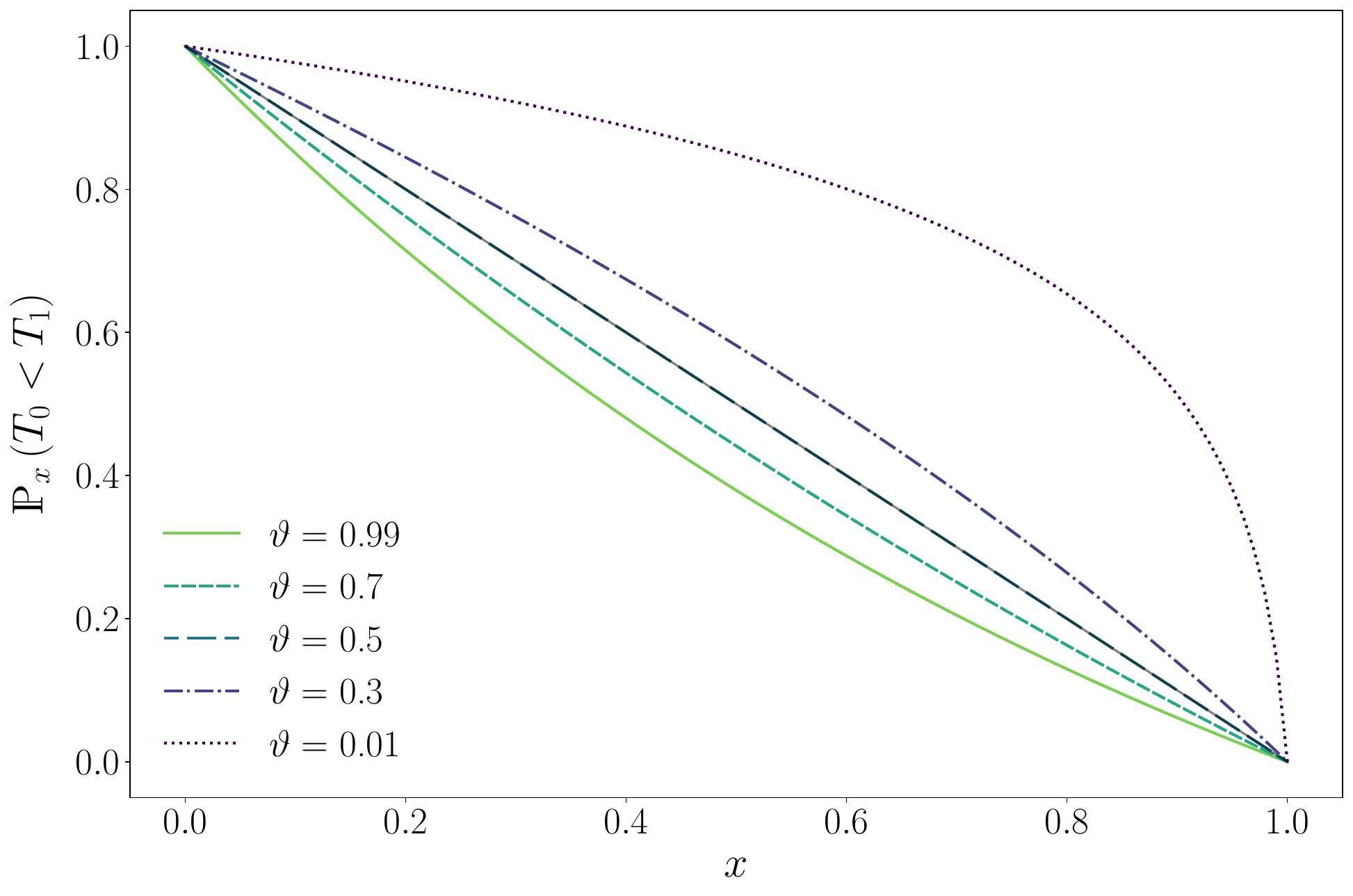} 
	\caption{$s = 0.5$}
	\label{subfig}
\end{subfigure}
\\[10pt]
\begin{subfigure}[c]{0.49\textwidth}
	\includegraphics[width=0.95\textwidth]{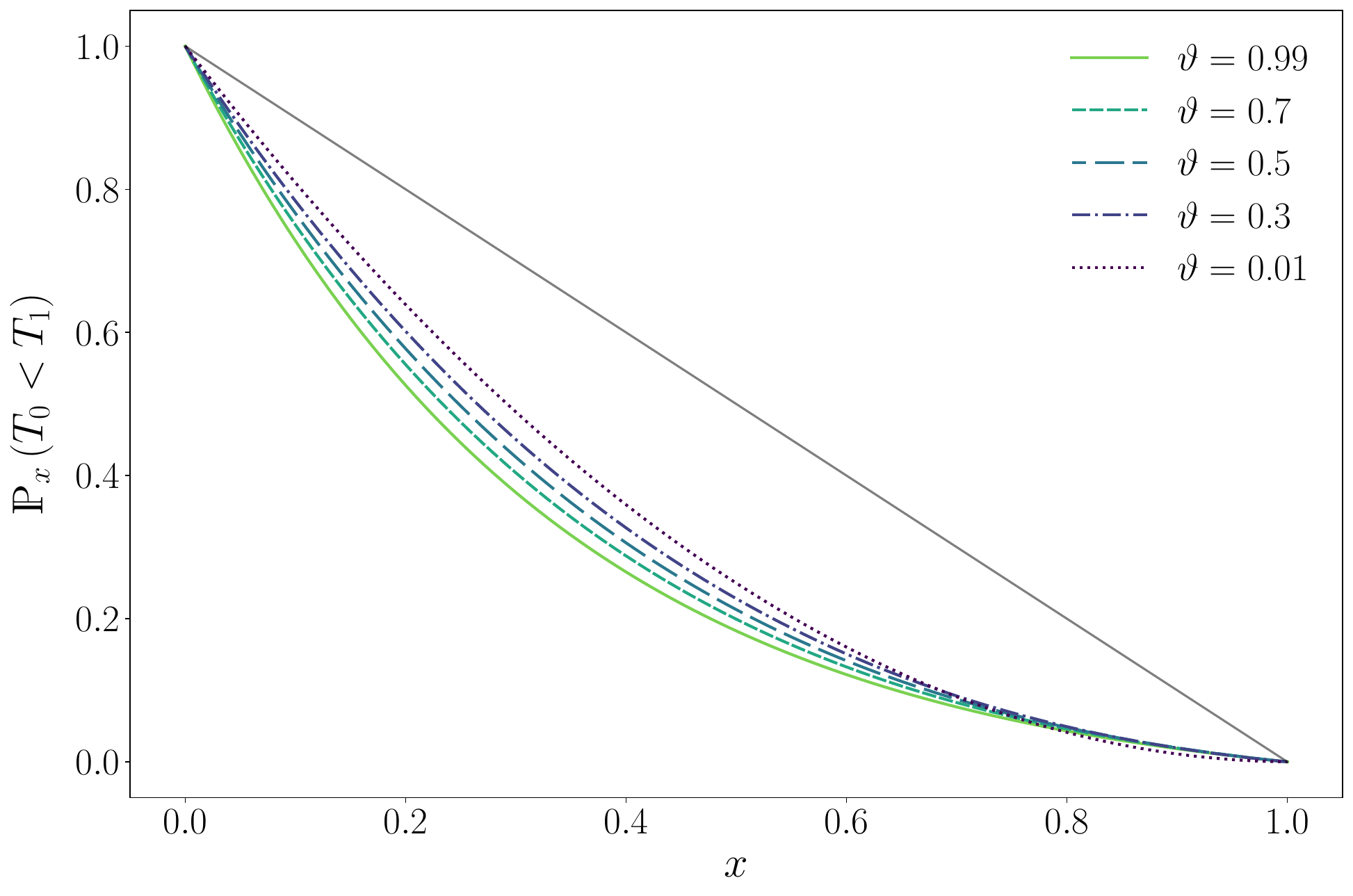}
	\caption{$s = 1.5$}
	\label{subfig1} 
\end{subfigure} 
\begin{subfigure}[c]{0.49\textwidth}
	\includegraphics[width=0.95\textwidth]{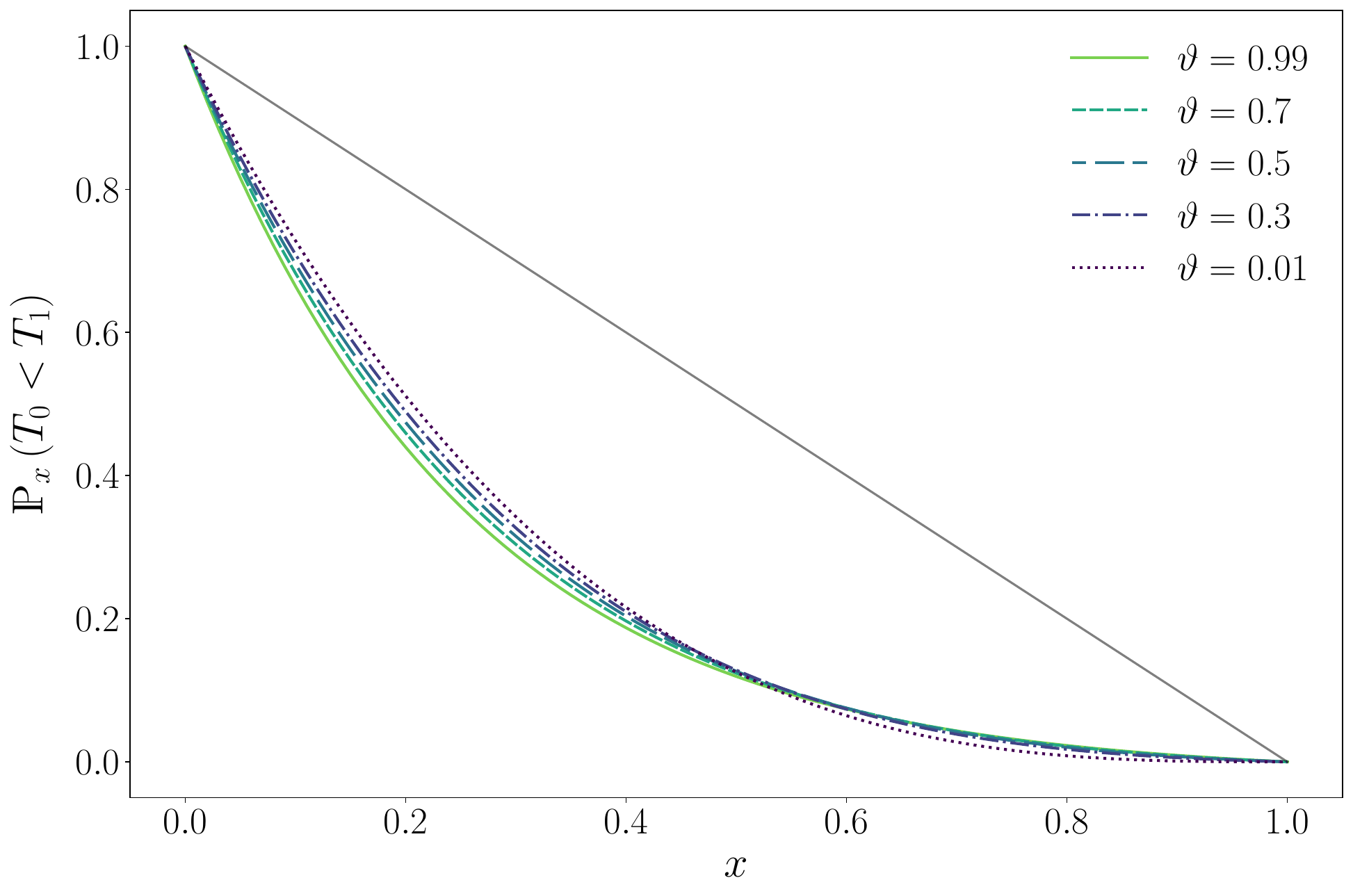}
	\caption{$s = 2$}
	\label{subfig2} 
\end{subfigure}
\caption{The extinction probability $\Pb_{x}(T_{0} < T_1)$ from \eqref{eq.extinction} for 
solutions of the SDE \eqref{eq.SDE} with $\rho(x) = s x(1 - x)$, for different values of $s$. For each $s$, a fixed set of values for $\del$ is considered. The function $f(x) = 1 - x$ (black, solid line) is plotted for reference.}
\label{fig.fixation2}
\end{figure}

Another quantity of interest in biological applications is the mean time to absorption, $\Eb_{x}[T_{0,1}]$, as a function of $x$, where $T_{0,1} \coloneqq T_{0} \wedge T_1$. This is given by
\begin{equation} \label{eq.fixation_def}
\Eb_{x}[T_{0,1}] = \int_{0}^{1} G(x,\nu)\, \mathrm{d}\nu,
\end{equation}
where $G(x, \nu)$ is the Green’s function, defined as
\begin{align*}
G(x, \nu) = 
\begin{cases}
\displaystyle 2\,\frac{S(1) - S(x)}{S(1) - S(0)} \cdot \frac{S(\nu) - S(0)}{\sigma^2(\nu)\, S'(\nu)} &\text{if } 0 < \nu < x, \\[6pt]
\displaystyle 2\,\frac{S(x) - S(0)}{S(1) - S(0)} \cdot \frac{S(1) - S(\nu)}{\sigma^2(\nu)\, S'(\nu)} &\text{if } x < \nu < 1,
\end{cases}
\end{align*}
with $S(x)$ the scale function from \eqref{eq.scale}.

Even in the case of genic selection, i.e., when $d(x) = (-(1 - \del) + s)x(1 - x)$, the integral in \eqref{eq.fixation_def} generally cannot be evaluated analytically. However, in the absence of selection ($s = 0$), a straightforward computation yields the explicit formula:
\begin{align}\label{eq.fix.analytic}
\Eb_{x}[T_{0,1}]\ =\ 2 \ln(1 - x) \cdot \frac{\del^{-1} - \big(1 - (1 - \del)x\big)^{-1}}{1 - \del^{-1}} + 2 \ln(x) \cdot \frac{1 - \big(1 - (1 - \del)x\big)^{-1}}{1 - \del}.
\end{align}
A plot of this expression is shown in Figure~\ref{fig.fix2}. Although the integral in \eqref{eq.fixation_def} cannot be solved in closed form for $s \neq 0$, it can be evaluated numerically. We provide such a plot for the case $s = 2$—where the extinction probabilities exhibit non-monotonic dependence on $\del$—in the same figure.
In both cases, one observes that the expected time to absorption increases as $\del$ decreases.

\begin{figure}[h!]\centering
\begin{subfigure}[t]{0.49\textwidth}
	\includegraphics[width=0.99\textwidth]{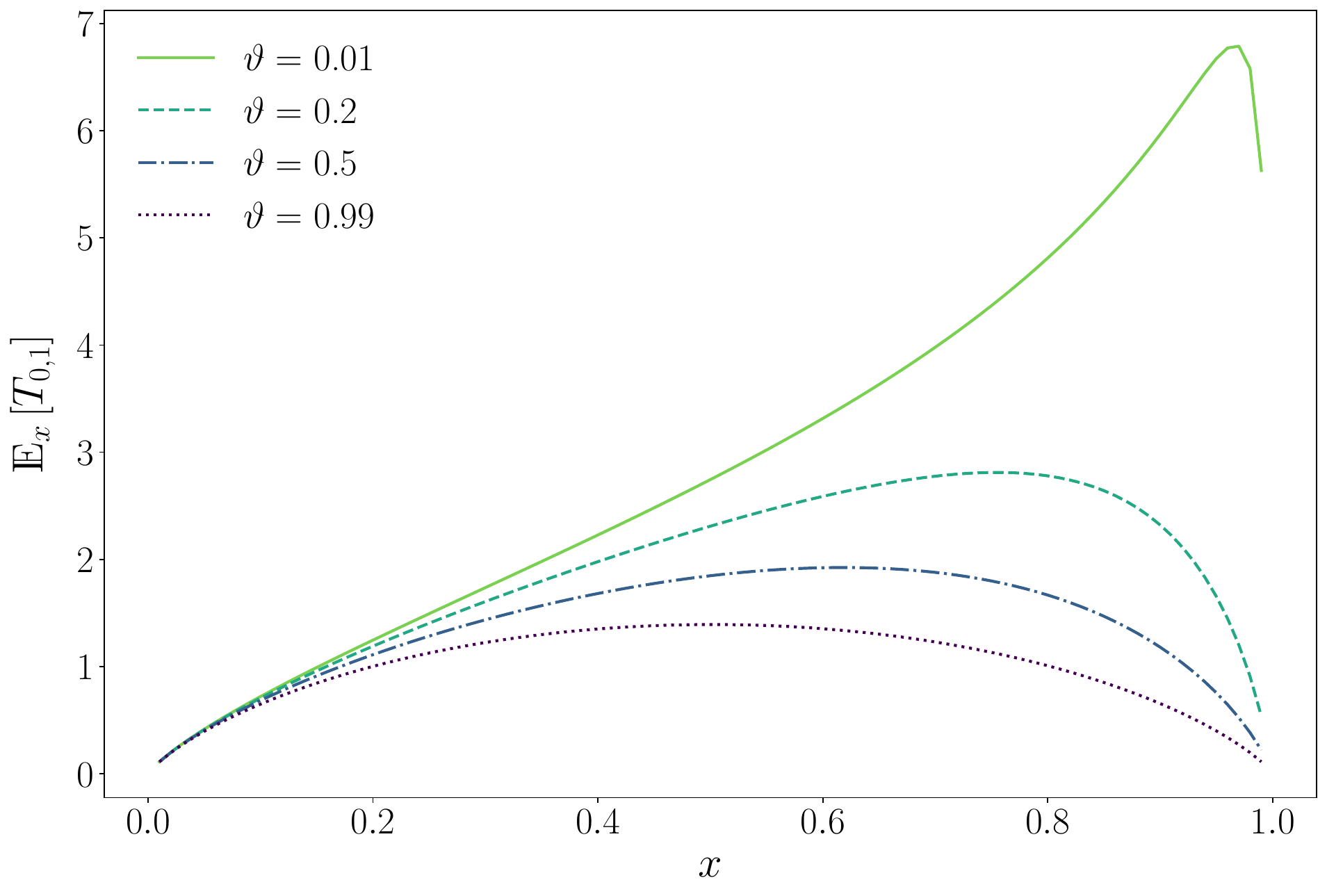}  
	\caption{Analytical result from \eqref{eq.fix.analytic} for the case $s = 0$.}
\end{subfigure}
\begin{subfigure}[t]{0.49\textwidth}
	\includegraphics[width=0.99\textwidth]{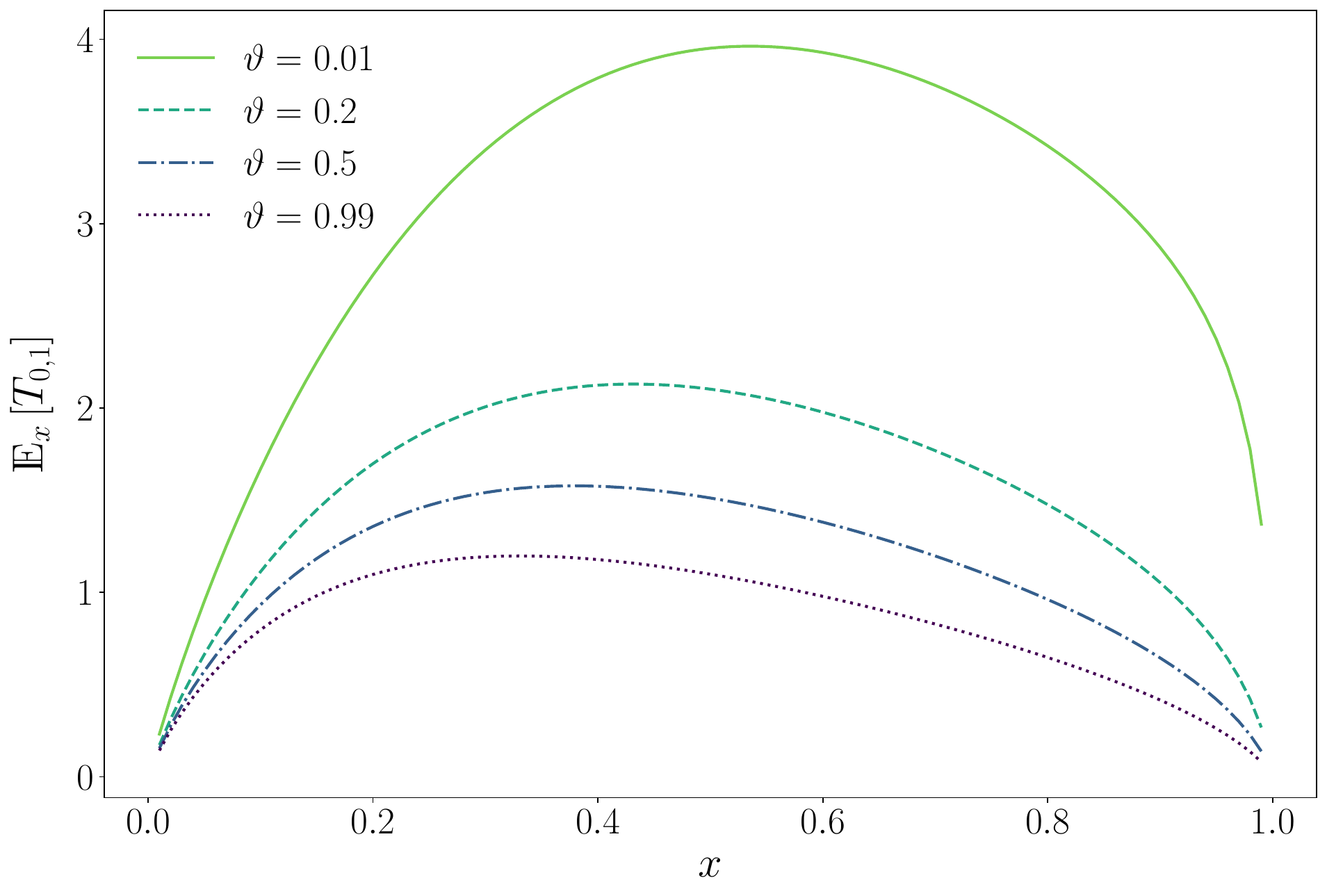} 
	\caption{Result for $s = 2$, obtained via numerical integration (using the \texttt{scipy} package).}
\end{subfigure}
\caption{Mean time to absorption $\Eb_{x}[T_{0,1}]$ for the solution of the SDE \eqref{eq.SDE} with drift term $d(x) = (-(1 - \del) + s)x(1 - x)$.}
\label{fig.fix2}
\end{figure}

\vspace{.1cm}
Finally, we briefly comment on the stationary distribution of $X$ in the case where $\beta_{0}, \beta_1 > 0$.  According to \cite[Theorem 3.24]{Etheridge:11}, the density of the stationary distribution is given by the ratio $m(x) / \int_{0}^1 m(x)\, \mathrm{d}x$, where
$$ m(x)\ \coloneqq\ \frac{1}{\sigma^2(x)\, S'(x)}, $$
and $S(x)$ is again the scale function. In the setting
$$ \rho(x)\ =\ \beta_{0}(1 - x) - \beta_1 x + s\, x(1 - x) $$
(which corresponds to genic selection favoring small individuals and bi-directional mutation), the density of the stationary distribution simplifies to
\begin{equation*}
C(\beta_{0}, \beta_1, \vartheta, s)\,
x^{2\beta_{0} - 1}\,
(1 - x)^{2\beta_1 \del^{-1} - 1}\,
\big(1 - (1 - \del)x\big)^{-2\beta_{0} - 2\beta_1 \del^{-1} - 2s(1 - \del)^{-1} + 1},
\end{equation*}
where $C(\beta_{0}, \beta_1, \vartheta, s) > 0$ is a normalizing constant. This constant can be expressed in terms of hypergeometric functions and evaluated numerically.\\

\textbf{Acknowledgments.}  
The authors are very grateful to two anonymous reviewers for their careful reading and insightful comments, which significantly improved the manuscript. We also thank Ellen Baake for numerous helpful discussions and valuable insights.  

Gerold Alsmeyer acknowledges financial support by the German Research Foundation (DFG) under Germany's Excellence Strategy EXC 2044--390685587, Mathematics M\"unster: Dynamics--Geometry--Structure. Fernando Cordero and Hannah Dopmeyer gratefully acknowledge financial support by the German Research Foundation (DFG) -- Project-ID 317210226 -- SFB1283.

\addtocontents{toc}{\protect\setcounter{tocdepth}{2}}
\bibliographystyle{abbrv}
\bibliography{reference}

\end{document}